\documentclass[12pt]{article}
\usepackage[margin=1in]{geometry}

\usepackage{amsmath}
\usepackage{amssymb}
\usepackage{amsthm}
\usepackage{enumitem}
\usepackage{tikz}
\usepackage{graphicx}
\usepackage{caption}
\usepackage{subcaption}
\usepackage{url}

\usepackage{thmtools}
\usepackage{thm-restate}
\usepackage{hyperref}

\newcommand{\ex}{\text{ex}}

\newtheorem{theorem}{Theorem}[section]
\newtheorem{lemma}[theorem]{Lemma}
\newtheorem{proposition}[theorem]{Proposition}

\newtheorem{claim}{Claim}[theorem]
\newtheorem{conjecture}[theorem]{Conjecture}
\newtheorem{question}[theorem]{Question}
\newtheorem{corollary}[theorem]{Corollary}

\newenvironment{subproof}[1][Proof]{\begin{proof}[#1]}{\end{proof}}

\title{Tight minimum degree conditions for apex-outerplanar minors and subdivisions in graphs and digraphs}
\author{Chun-Hung Liu\thanks{chliu@tamu.edu. Department of Mathematics, Texas A\&M University, USA. Partially supported by NSF under award DMS-1954054 and CAREER award DMS-2144042.}
\and 
Youngho Yoo\thanks{yyoo2@alaska.edu. Department of Mathematics and Statistics, University of Alaska Fairbanks, USA.}}
\date{\today}

\begin{document}

\maketitle

\abstract{
Motivated by Hadwiger's conjecture and related problems for list-coloring, we study graphs $H$ for which every graph with minimum degree at least $|V(H)|-1$ contains $H$ as a minor.
We prove that a large class of apex-outerplanar graphs satisfies this property.
Our result gives the first examples of such graphs whose vertex cover numbers are significantly larger than half of the number of its vertices, which breaks a barrier for attacking related coloring problems via extremal functions, and recovers all known such graphs that have arbitrarily large maximum degree.
Our proof can be adapted to directed graphs to show that if $\vec H$ is the digraph obtained from a directed cycle or an in-arborescence by adding an apex source, then every digraph with minimum out-degree $|V(\vec H)|-1$ contains $\vec H$ as a subdivision or a butterfly minor respectively.
These results provide the optimal upper bound for the chromatic number and dichromatic number of graphs and digraphs that do not contain the aforementioned graphs or digraphs as a minor, butterfly minor and a subdivision, respectively. 
Special cases of our results solve an open problem of Aboulker, Cohen, Havet, Lochet, Moura and Thomass\'{e} and strengthen results of Gishboliner, Steiner and Szab\'{o}.
}

\section{Introduction}

A graph\footnote{All graphs in this paper are simple.} $H$ is a \emph{minor} of another graph $G$ if $H$ is isomorphic to a graph that can be obtained from a subgraph of $G$ by contracting edges.
Seymour proposed the following question in his survey \cite{seymour2016hadwiger} on Hadwiger's conjecture.

\begin{question}[\cite{seymour2016hadwiger}] \label{q:hadwiger}
For which graphs $H$ does it hold that every graph with chromatic number at least $|V(H)|$ contains $H$ as a minor? 
\end{question}

The famous conjecture of Hadwiger \cite{hadwiger1943klassifikation} states that every complete graph satisfies Question \ref{q:hadwiger}.
Since every graph $H$ is a spanning subgraph of the complete graph on $|V(H)|$ vertices, Hadwiger's conjecture is equivalent to the statement that every graph satisfies Question \ref{q:hadwiger}.
Moreover, every graph $H$ on at most six vertices satisfies Question \ref{q:hadwiger} since Hadwiger's conjecture is known to be true for complete graphs with up to six vertices \cite{hadwiger1943klassifikation, wagner1937eigenschaft, appel1977every, robertson1997four, robertson1993hadwiger}.

Beyond graphs on six vertices, Question \ref{q:hadwiger} is only known to hold for few special cases, and the proofs of most of these cases rely on average degree arguments. 
A simple greedy algorithm shows that if a graph $G$ is \emph{$d$-degenerate} (that is, every subgraph of $G$ has a vertex of degree at most $d$), then $G$ has chromatic number at most $d+1$. 
This implies that every graph $G$ with chromatic number at least $t$ contains a subgraph with minimum degree at least $t-1$, hence average degree at least $t-1$.
Therefore, if $\ex_H(n)<\frac{|V(H)|-1}{2}n$, where $\ex_H(n)$ denotes the maximum number of edges of a graph on $n$ vertices with no $H$ minor, then every graph $G$ with  no $H$ minor has average degree less than $|V(H)|-1$ and is hence $(|V(H)|-2)$-degenerate and $(|V(H)|-1)$-colorable.
Examples of graphs $H$ satisfying the condition $\ex_H(n)<\frac{|V(H)|-1}{2}n$ include cycles \cite{erdos1959maximal} and $K_{2,t}$ for all $t\geq 1$ \cite{chudnovsky2011edge}. 
(Woodall \cite{woodall2001list} proved earlier, without the use of extremal functions, that $K_{2,t}$ satisfies Question \ref{q:hadwiger} and its strengthening to list-coloring).

For some other graphs $H$, if $\ex_H(n) < \frac{|V(H)|+c}{2}n$ for some small constant $c$, then it can sometimes still be shown that $H$ satisfies Question \ref{q:hadwiger} with additional arguments. 
This approach has been used successfully for some graphs on 7, 8, and 9 vertices  \cite{jakobsen1970homomorphism,lafferty2022K8,lafferty2022K9}, for the Petersen graph \cite{hendrey2018extremal}, and for $K_{s,t}^*$ (obtained from $K_{s,t}$ by adding all edges joining pairs of vertices in the part of size $s$ in the bipartition) if $t$ is sufficiently large relative to $s$ \cite{kostochka2010ks,kostochka2014minors}. 
It does not seem to be known whether these graphs also satisfy the strengthening of Question \ref{q:hadwiger} to list-coloring or correspondence-coloring.

While average degree arguments have been successful in answering Question \ref{q:hadwiger} for these few cases, there is an inherent limitation to this approach, namely that it is not very helpful for graphs $H$ with large vertex cover number.
A \emph{vertex cover} of a graph $G$ is a set of vertices meeting every edge, and the \emph{vertex cover number} of $G$, denoted by $\tau(G)$, is the minimum size of a vertex cover of $G$. 
Note that $\tau$ is a minor-monotone parameter; that is, if $G'$ is a minor of $G$, then $\tau(G')\leq \tau(G)$.
It follows that $H$ is not a minor of $G=K_{\tau(H)-1,n-\tau(H)+1}$ since $\tau(G)\leq\tau(H)-1$.
On the other hand, we have $\ex_H(n)\geq |E(G)| = (\tau(H)-1)(n-\tau(H)+1) = (\tau(H)-1)n - (\tau(H)-1)^2$.
Hence, the aforementioned average degree approach for Question \ref{q:hadwiger} requires that $\tau(H)-1\leq \frac{|V(H)|+c}{2}$ for some small constant $c$.
In fact, to our knowledge, all graphs $H$ currently known to satisfy Question \ref{q:hadwiger} have $\tau(H) \leq \frac{|V(H)|}{2} + O(1)$. 

In this paper, we study minimum degree conditions that force $H$ as a minor, which gives an alternative approach for Question \ref{q:hadwiger} without being limited by vertex covers. 
Our results on minors give new results on forcing a subdivision of $H$.
Furthermore, our proof can be generalized to directed graphs to give optimal bounds on the minimum out-degree of a directed graph that force a butterfly minor or a subdivision of directed graphs $\vec H$.
Special cases of our results on directed graphs resolve an open problem of Aboulker et al.\ \cite{aboulker2019subdivisions} and strengthen some results of Gishboliner et al.\ in \cite{gishboliner2022dichromatic} and \cite{gishboliner2022oriented}.

\subsection{Minors in undirected graphs}

We begin by considering the following question proposed by Seymour \cite{seymour2017birs}.

\begin{question}[\cite{seymour2017birs}] \label{q:mindegree}
For which graphs $H$ does it hold that every graph with minimum degree at least $|V(H)|-1$ contains $H$ as a minor? 
\end{question}

Since every graph with chromatic number at least $t$ contains a subgraph with minimum degree at least $t-1$, every graph $H$ satisfying Question \ref{q:mindegree} also satisfies Question \ref{q:hadwiger}.
Also note that this minimum degree bound cannot be improved since $K_{|V(H)|-1}$ has minimum degree $|V(H)|-2$ but does not contain $H$ as a minor.

Not all graphs satisfy Question \ref{q:mindegree}.
For example, $K_{3,3}$ and $K_5$ do not satisfy Question \ref{q:mindegree} due to the existence of infinitely many planar graphs with minimum degree five.
By adding apices to such planar graphs, it follows that the graphs $K_{3,3}\lor K_s$ (the graph obtained from a disjoint union of $K_{3,3}$ and $K_s$ by adding all edges between them) and $K_t$ with $t\geq 5$ also do not satisfy Question \ref{q:mindegree}.

Some positive answers for Question \ref{q:mindegree} are known.
Straightforward arguments shows that trees and cycles satisfy Question \ref{q:mindegree}.
Other graphs that satisfy Question \ref{q:mindegree} include: wheels \cite{turner2005}, graphs obtained from a star by adding a new vertex adjacent to all other vertices \cite{mader1971}, and graphs obtained from the 5-vertex path \cite{pelikan1968} (and the $k$-vertex path for $k \in \{6,7\}$ \cite{diwan1971}, respectively) by adding edges between every pair of vertices with distance at most three.
Recently, Norin and Turcotte \cite{norin2025limits} showed that for every graph family $\mathcal{F}$ with strongly sublinear separators, every sufficiently large bipartite graph in $\mathcal{F}$ with bounded maximum degree satisfies Question \ref{q:mindegree}.

Note that all graphs known to satisfy Question \ref{q:mindegree} have vertex cover number at most $\frac{|V(H)|}{2}+O(1)$. 
Among those graphs that have unbounded maximum degree besides trees, there is a vertex adjacent to all other vertices, and only few vertices have unbounded maximum degree.

Given a graph $H$, let $H^+$ denote the graph obtained from $H$ by adding a new vertex, called the \emph{apex}, adjacent to every vertex in $H$.
In this paper, we focus on the special case of Question \ref{q:mindegree} restricted to graphs with an apex.

\begin{question} \label{q:apexmindegree}
For which graphs $H$ does it hold that every graph with minimum degree at least $|V(H^+)|-1$ contains $H^+$ as a minor? 
\end{question}

Again, observe that $K_{2,3}$ and $K_4$ do not satisfy Question \ref{q:apexmindegree} due to the existence of planar graphs with minimum degree five. 
It is thus natural to ask whether outerplanar graphs $H$ satisfy Question \ref{q:apexmindegree}.

The first main result of this paper is that Question \ref{q:apexmindegree} holds for a large class of outerplanar graphs that we call ``contractibly orderable,'' which includes graphs that have vertex cover number significantly greater than $\frac{|V(H)|}{2}+O(1)$ and have unbounded number of vertices of unbounded maximum degree.

An \emph{$H$-minor} in a graph $G$ is a set $\mu=\{\mu(v):v\in V(H)\}$ of vertex-disjoint subsets $\mu(v)$ of $V(G)$ such that $G[\mu(v)]$ is connected for all $v\in V(H)$ and for each edge $xy\in E(H)$, there is an edge of $G$ with one end in $\mu(x)$ and the other end in $\mu(y)$.
The set $\mu(v)$ is called the \emph{branch set} of $v\in V(H)$. Clearly, $H$ is a minor of $G$ if and only if $G$ contains an $H$-minor.

\begin{restatable}{theorem}{mainmindeg}
\label{thm:mainmindeg}
If $H$ is a spanning subgraph of a contractibly orderable graph, then every graph with minimum degree at least $|V(H^+)|-1$ contains an $H^+$-minor such that the branch set of the apex is a singleton.    
\end{restatable}

We defer precise definitions to Section \ref{sec:contractibleorderings}, but let us note here that contractibly orderable graphs include many natural outerplanar graphs.
Cacti are very special cases of contractibly orderable graphs (Proposition \ref{prop:cactus}), so Theorem \ref{thm:mainmindeg} affirmatively answers Question \ref{q:mindegree} for wheels and apex-trees, which include all known positive answers to Question \ref{q:mindegree} with unbounded maximum degree.
More generally, contractibly orderable graphs include every maximal outerplanar graph whose weak dual tree is a path (Proposition \ref{prop:weakdualpath}), a complete cubic tree (Proposition \ref{prop:universal}), or a complete binary tree (Proposition \ref{prop:binarytree}).

\begin{corollary} \label{cor:orderex_into}
If $H$ is a cactus or a spanning subgraph of a maximal outerplanar graph whose weak dual tree is either a path, a complete cubic tree, or a complete binary tree, then every graph with minimum degree at least $|V(H^+)|-1$ contains an $H^+$-minor such that the branch set of the apex is a singleton.
\end{corollary}

Corollary \ref{cor:orderex_into} exhibits the first infinite class of graphs $H$ proven to satisfy Question \ref{q:hadwiger} with vertex cover number significantly greater than $\frac{|V(H)|}{2} + O(1)$. 
For example, let $s\geq 2$ be an integer and let $H$ be the (unique, up to isomorphism) maximal outerplanar graph on $3s$ vertices with maximum degree 4.
Then $H$ has $s$ vertex-disjoint triangles.
Since every vertex cover of $H$ contains at least two vertices of each triangle, we have $\tau(H^+)\geq 2s+1 \geq \frac23|V(H^+)|$. 
On the other hand, the weak dual tree of $H$ is a path, so $H^+$ satisfies Question \ref{q:hadwiger} by Corollary \ref{cor:orderex_into}.

One can ask an analogous problem of Question \ref{q:hadwiger} for list-coloring and correspondence-coloring (which is a generalization of list-coloring and is also known as DP-coloring). See \cite{bernshteyn2017dp,dvovrak2018correspondence} for the definitions of list-coloring and correspondence-coloring.
The list-coloring variant of Hadwiger's conjecture, that every graph with list-chromatic number at least $t$ contains a $K_t$-minor, is false: Voigt \cite{voigt1993list} showed that there exist planar graphs that are not 4-list-colorable, and Bar\'at, Joret and Wood \cite{barat2011disproof} showed that there exist graphs with no $K_{3t+2}$-minor that are not $4t$-list-colorable. 

On the other hand, Theorem \ref{thm:mainmindeg} gives exact results for the list-coloring and the correspondence-coloring variants of Question \ref{q:hadwiger}.
The greedy algorithm also shows that every $d$-degenerate graph is $(d+1)$-list-colorable and, more generally, $(d+1)$-correspondence-colorable. 
Hence Theorem \ref{thm:mainmindeg} immediately implies the following corollary. 

\begin{corollary}
If $H$ is a spanning subgraph of a contractibly orderable graph, then every graph with no $H^+$-minor is $(|V(H^+)|-1)$-list-colorable and $(|V(H^+)|-1)$-correspondence-colorable.
\end{corollary}

\subsection{Subdivisions in undirected graphs}

It is well-known that if $G$ contains an $H$-minor such that the branch set of $v\in V(H)$ is a singleton for all $v$ with degree at least 4 in $H$, then $G$ contains a subdivision of $H$.
Since the branch set of the apex is a singleton in Theorem \ref{thm:mainmindeg}, this implies, for example, that every graph with minimum degree $t$ contains a subdivision of the wheel on $t+1$ vertices, recovering the result of Turner \cite{turner2005}.
This motivates the following analogues of Questions \ref{q:hadwiger} and \ref{q:mindegree} for subdivisions.

\begin{question} \label{q:subdivchro}
For which graphs $H$ does it hold that every graph with chromatic number at least $|V(H)|$ contains a subdivision of $H$?
\end{question}

\begin{question} \label{q:subdivmindeg}
For which graphs $H$ does it hold that every graph with minimum degree at least $|V(H)|-1$ contains a subdivision of $H$?
\end{question}

Since Haj\'{o}s' conjecture, the analogue of Hadwiger's conjecture with respect to subdivisions, is known to be false, Question \ref{q:subdivchro} is more ``meaningful'' than Question \ref{q:hadwiger} in the sense that not all graphs satisfy Question \ref{q:subdivchro}, whereas it is unknown whether there is a graph that does not satisfy Question \ref{q:hadwiger}.

Some aforementioned known positive answers for Question \ref{q:mindegree} are actually also positive answers to Question \ref{q:subdivmindeg}, namely trees and cycles,
wheels \cite{turner2005}, apex-stars \cite{mader1971}, and the graphs obtained from the $k$-vertex path (for $k \in \{5,6,7\}$) by adding edges between every pair of vertices with distance at most three \cite{pelikan1968,diwan1971}.

Theorem \ref{thm:mainmindeg} immediately implies the following positive answers for Question \ref{q:subdivmindeg} and hence for Question \ref{q:subdivchro}. A graph is \emph{subcubic} if every vertex has degree at most 3.

\begin{corollary} \label{cor:D2subdiv_intro}
Let $H$ be a subcubic graph that is a spanning subgraph of a contractibly orderable graph and let $H^*$ denote the graph obtained from $H$ by adding a new vertex $a$ and new edges $av$ for every vertex $v\in V(H)$ with degree at most 2 in $H$. 
Then every graph with minimum degree at least $|V(H^*)|-1$ contains a subdivision of $H^*$.
\end{corollary}
In particular, if $H$ has maximum degree at most 2, then $H^+$ satisfies Question \ref{q:subdivmindeg}.
Note that this includes wheels.

\subsection{Butterfly minors in directed graphs}
The proof of Theorem \ref{thm:mainmindeg} can be adapted to yield more general results for directed graphs (henceforth \emph{digraphs}\footnote{All digraphs in this paper are simple. That is, every digraph does not contain loops and does not contain two distinct arcs with the same tails and heads. Note that (simple) digraphs can contain two arcs that have the same pair of ends as long as they have different directions.}).
The \emph{dichromatic number} of a digraph $\vec{G}=(V,E)$ is an analogue of the chromatic number, defined to be the smallest integer $t$ such that there exists a partition of $V$ into $t$ parts $V_1,\dots,V_t$ such that $\vec{G}[V_i]$ has no directed cycle for each $i\in[t]$.
Again, the greedy algorithm can be applied to show that every digraph with dichromatic number at least $t+1$ contains a subdigraph with minimum out-degree (and in-degree) at least $t$. 

Let us first consider the directed analogues of Questions \ref{q:hadwiger} and \ref{q:mindegree}. One popular notion of minors of digraphs is the butterfly minor:
A digraph $\vec H$ is a \emph{butterfly minor} of a digraph $\vec G$ if a digraph isomorphic to $\vec H$ can be obtained from $\vec G$ by a sequence of the following operations: deleting a vertex, deleting an arc, and contracting an arc $(u,v)$ where $u$ has out-degree 1 or $v$ has in-degree 1 (and deleting all resulting loops and parallel arcs).
In this case, for each $v\in V(\vec H)$, the \emph{branch set} of $v$ is the set of vertices of $\vec G$ that were contracted into the vertex corresponding to $v$.
Note that if $\vec H$ is a butterfly minor of $\vec G'$ and $\vec G'$ is a butterfly minor of $\vec G$, then $\vec H$ is a butterfly minor of $\vec G$.
\begin{question}
    \label{q:dichrobutmin}
    For which digraphs $\vec H$ does it hold that every digraph with dichromatic number at least $|V(\vec H)|$ contains $\vec H$ as a butterfly minor?
\end{question}
\begin{question}
    \label{q:outdegbutmin}
    For which digraphs $\vec H$ does it hold that every digraph with minimum out-degree at least $|V(\vec H)|-1$ contains $\vec H$ as a butterfly minor?
\end{question}

The \emph{biorientation} of a graph $G$ is the digraph obtained by replacing each edge of $G$ by two arcs with opposite directions.
It is easy to see that the chromatic number of $G$ equals the dichromatic number of $\vec{G}$, and the minimum degree of $G$ equals the minimum out-degree of $\vec{G}$.
Moreover, if $\vec{G}$ contains a digraph $\vec H$ as a subdigraph, then $G$ contains the underlying graph of $\vec H$ as a subgraph. 
So positive answers to Questions \ref{q:dichrobutmin} and \ref{q:outdegbutmin} give positive answers to Questions \ref{q:subdivchro} and \ref{q:subdivmindeg}, respectively.

Again, note that the bounds on the dichromatic number and the minimum out-degree in Questions \ref{q:dichrobutmin} and \ref{q:outdegbutmin} cannot be improved since the biorientation of $K_{|V(\vec H)|-1}$ has dichromatic number $|V(\vec H)|-1$ and minimum out-degree $|V(\vec H)|-2$, but does not contain $\vec H$ as a butterfly minor.

An \emph{in-arborescence} (respectively, \emph{out-arborescence}) is a directed tree with a vertex, called its \emph{root}, such that every arc is oriented towards (respectively, away from) the root. 
Given a digraph $\vec H$, let $\vec H^+$ denote the digraph obtained from $\vec H$ by adding a new vertex, called the \emph{apex source}, and new arcs from the new vertex to all other vertices.
Specializing our proof of Theorem \ref{thm:mainmindeg} to apex-trees immediately yields the following positive answer to Question \ref{q:outdegbutmin} (and hence to Question \ref{q:dichrobutmin}).

\begin{restatable}{theorem}{butterflydeg}
\label{thm:butterfly_deg}
If $\vec H$ is an in-arborescence, then every digraph with minimum out-degree at least $|V(\vec H^+)|-1$ contains $\vec H^+$ as a butterfly minor such that the branch set of the apex source is a singleton.  
\end{restatable}

\subsection{Subdivisions in directed graphs}
Theorem \ref{thm:butterfly_deg} has interesting consequences for subdivisions.
Subdivisions of digraphs are significantly more difficult to obtain compared to the undirected setting.
For example, it is known that every (undirected) graph $H$ is a subdivision of every graph with sufficiently large minimum degree \cite{mader1967homomorphieeigenschaften}, whereas there exist digraphs with arbitrarily large minimum out-degree and in-degree that do not contain the biorientation of $K_3$ as a subdivision \cite{thomassen1985even}.
The following conjecture of Mader from 1985 remains open.

\begin{conjecture}[\cite{mader1985degree}]\label{conj:mader}
If $\vec H$ is an acyclic digraph, then there exists an integer $d(\vec H)$ such that every digraph with minimum out-degree at least $d(\vec H)$ contains a subdivision of $\vec H$.
\end{conjecture}

It is straightforward to show that out-arborescences satisfy Conjecture \ref{conj:mader}.
Mader \cite{mader1996topological} showed in 1996 that every digraph with minimum out-degree at least 3 contains a subdivision of the transitive tournament on 4 vertices.
It is not known whether the transitive tournament on 5 vertices satisfies Conjecture \ref{conj:mader}.
Recently, in-arborescences, oriented paths \cite{aboulker2019subdivisions}, and oriented cycles \cite{gishboliner2022oriented} were shown to satisfy Conjecture \ref{conj:mader}.

Note that Conjecture \ref{conj:mader} only considers the existence of the function $d(\vec H)$.
In this paper, we give new positive answers to Conjecture \ref{conj:mader}, with optimal function $d(\vec H)$, by considering the following directed analogues of Questions \ref{q:subdivchro} and \ref{q:subdivmindeg}.

\begin{question} \label{q:subdivchro_di}
For which digraphs $\vec{H}$ does it hold that every digraph with dichromatic number at least $|V(\vec H)|$ contains a subdivision of $\vec{H}$?
\end{question}

\begin{question} \label{q:subdivmindeg_di}
For which digraphs $\vec{H}$ does it hold that every digraph with minimum out-degree at least $|V(\vec H)|-1$ contains a subdivision of $\vec{H}$?
\end{question}

It is known that Question \ref{q:subdivmindeg_di} holds for oriented paths and directed cycles \cite{aboulker2019subdivisions} and that Question \ref{q:subdivchro_di} holds for orientations of cacti and biorientations of forests \cite{gishboliner2022dichromatic}.

Let us now discuss an application of Theorem \ref{thm:butterfly_deg}.
A vertex $v$ in a digraph $\vec H$ is called \emph{subcubic} if $v$ has total degree at most 3 and has in-degree and out-degree at most 2 in $\vec H$. 
It is not difficult to see that, in a butterfly minor, every branch set is either an in-arborescence, an out-arborescence, or the vertex-disjoint union of an in-arborescence and an out-arborescence joined by an arc from the root of the in-arborescence to the root of the out-arborescence.
This implies that if $\vec H$ is a butterfly minor of $\vec G$ and the branch set of $v\in V(\vec H)$ is a singleton for all $v$ that is \emph{not} subcubic in $\vec H$, then $\vec G$ contains a subdivision of $\vec H$.
Since the branch set of the apex source is a singleton in Theorem \ref{thm:butterfly_deg}, we immediately obtain the following positive answer to Question \ref{q:subdivmindeg_di} (and hence to Question \ref{q:subdivchro_di} and to Conjecture \ref{conj:mader}), analogously to Corollary \ref{cor:D2subdiv_intro}.
Note that a vertex in an in-arborescence is subcubic if and only if it has in-degree at most 2.

\begin{corollary}\label{cor:inarborsubdiv}
    Let $\vec H$ be an in-arborescence such that every vertex has in-degree at most 2. Let $\vec H^*$ denote the digraph obtained from $\vec H$ by adding a new vertex $a$ and new arcs $(a,v)$ for every vertex $v\in V(\vec H)$ with in-degree at most 1 in $\vec H$.
    Then every digraph with minimum out-degree at least $|V(\vec H^*)|-1$ contains a subdivision of $\vec H^*$.
\end{corollary}

Corollary \ref{cor:inarborsubdiv} resolves a problem of Aboulker et al. \cite[Problem 25]{aboulker2019subdivisions} on oriented cycles with two blocks, in a stronger form.
Let $C(k_1,k_2)$ denote the union of two internally disjoint directed paths, one of length $k_1$ and one of length $k_2$, with the same initial vertex $a$ and same terminal vertex $b$.
The problem of determining the optimal minimum out-degree condition to guarantee a subdivision of $C(k_1,k_2)$ was posed in \cite[Problem 25]{aboulker2019subdivisions}, where it was proven that every digraph with minimum out-degree at least $2(k_1+k_2)-1$ contains a subdivision of $C(k_1,k_2)$.
In \cite{gishboliner2022oriented}, this bound was improved to $k_1+3k_2-5$ if $k_1\geq k_2\geq 2$.
Observe that $C(k_1,k_2)\setminus a$ is an in-arborescence with every vertex having in-degree 1 except $b$, which has in-degree 2.
Therefore, Corollary \ref{cor:inarborsubdiv} immediately implies the following corollary, which provides the optimal answer to \cite[Problem 25]{aboulker2019subdivisions} in a stronger form.

\begin{corollary}
    Let $\vec H^*$ denote the digraph obtained from $C(k_1,k_2)$ by adding arcs from $a$ to every vertex other than $b$.
    Then every digraph with minimum out-degree at least $k_1+k_2-1 = |V(\vec H^*)|-1$ contains a subdivision of $\vec H^*$.
\end{corollary}

With a slight modification of the proofs of Theorems \ref{thm:mainmindeg} and \ref{thm:butterfly_deg}, we further obtain certain orientations of wheels as positive answers to Question \ref{q:subdivmindeg_di}, which also implies the result of Turner \cite{turner2005} on undirected wheels. 

Let $\vec C_t$ denote the directed cycle on $t$ vertices, and recall that $\vec C_t^+$ is obtained from $\vec C_t$ by adding an apex source.
Define $\vec W_t^1$ to be the digraph obtained from $\vec C_t^+$ by adding an additional arc from one vertex on the cycle to the apex source.
Let us also define $\vec W_t^2$ to be the digraph obtained from $\vec C_t^+$ by reversing the orientation of one arc from apex source to the cycle.

We show that $\vec C_t^+$ and $\vec W_t^2$ satisfy Question \ref{q:subdivmindeg_di} (and hence Question \ref{q:subdivchro_di}). 
Observe that $\vec C_t^+$ and $\vec W_t^2$ are two orientations of the undirected wheel on $t+1$ vertices, and that both $\vec W_t^1$ and $\vec W_{t+1}^2$ contain subdivisions of both $\vec C_t^+$ and $\vec W_t^2$.

\begin{restatable}{theorem}{subdivwheel}
    \label{thm:subdiv_wheel_mid_deg_intro}
    Let $t\geq 2$ be an integer and let $\vec G$ be a digraph with minimum out-degree at least $t$.
    Then $\vec G$ contains either a subdivision of $\vec W_t^1$ or a subdivision of $\vec W_{t+1}^2$.
    In particular, $\vec G$ contains both a subdivision of $\vec C_t^+$ and a subdivision of $\vec W_t^2$.
\end{restatable}

We conclude this section with an application of Theorem \ref{thm:subdiv_wheel_mid_deg_intro} to orientations of $K_4$.
Recall Mader's result \cite{mader1996topological} that every digraph with minimum out-degree at least 3 contains a subdivision of the transitive tournament on 4 vertices.
Based on Mader's result, it was further shown in \cite{gishboliner2022dichromatic} that every digraph with dichromatic number at least 4 contains every orientation of $K_4$ by verifying non-transitive orientations of $K_4$.

Observe that there are (up to isomorphism) two orientations of $K_4$ that do not contain a \emph{sink} (a vertex with out-degree 0), and that these two orientations are isomorphic to either $\vec C_3^+$ or $\vec W_3^2$.

\begin{corollary}\label{cor:outdegK4}
    Every digraph with minimum out-degree at least $3$ contains a subdivision of every orientation of $K_4$ with no sink.
\end{corollary}

Reversing the directions of all arcs in Corollary \ref{cor:outdegK4}, we obtain that every digraph with minimum in-degree at least 3 contains a subdivision of every orientation of $K_4$ with no source.
Since every digraph with dichromatic number at least $t$ contains a subdigraph with minimum out-degree and minimum in-degree at least $t-1$, and since the only orientation of $K_4$ with both a source and a sink is the transitive tournament, Corollary \ref{cor:outdegK4}, together with Mader's result \cite{mader1996topological}, gives the following strengthening of a main result in \cite{gishboliner2022dichromatic}, in fact with a much shorter proof.

\begin{corollary}\label{cor:outindegK4}
    Every digraph with minimum out-degree and in-degree at least 3 contains a subdivision of every orientation of $K_4$.
\end{corollary}

\subsection{Paper outline}
In Section \ref{sec:contractibleorderings}, we prove our results on undirected graphs. 
We first outline our proof and define contractible orderings, then in Section \ref{sec:proofcontractibleordering} we prove Theorem \ref{thm:mainmindeg}.
In Section \ref{sec:COexamples}, we give several examples of contractibly orderable graphs. 
In Section \ref{sec:digraphs}, we prove our results on digraphs, beginning with the proof of Theorem \ref{thm:butterfly_deg} in Section \ref{sec:apexarborescence}, then proceeding to the proof of Theorem \ref{thm:subdiv_wheel_mid_deg_intro} in Section \ref{sec:directedwheel}.
Section \ref{sec:conclusion} contains some concluding remarks and conjectures.

\section{Contractible orderings} \label{sec:contractibleorderings}

Let $H$ be a graph.
For an edge $e=uv\in E(H)$, we write $H/e$ to denote the graph obtained from $H$ by contracting $e$ and deleting resulting loops and parallel edges; in $H/e$, both $u$ and $v$ refer to the same contracted vertex.
If $x$ is either a vertex or an edge of $H$, then we write $H\setminus x$ to denote the graph obtained from $H$ by deleting $x$.
An \emph{ordering} of $H$ is a sequence $\omega=(\omega_1,\dots,\omega_t)$ such that $V(H)=\{\omega_1,\dots,\omega_t\}$ and $|V(H)|=t$.
Given a subset $S\subseteq [t]$, we denote by $\omega_S$ the set $\{\omega_s:s\in S\}$.

Theorem \ref{thm:mainmindeg} involves particular kinds of orderings, which we call contractible orderings.
Before giving its precise definition, let us first describe the intuition by outlining our proof. 

Suppose we wish to construct an $H^+$-minor in a graph $G$ of minimum degree $t=|V(H)|$. We choose an ordering $\omega=(\omega_1,\dots,\omega_t)$ of $H$ and arbitrarily choose an initial singleton $S\subseteq V(G)$. Note that $G[S]\cong H[\{\omega_1\}]=H[\omega_{\{1\}}]$. 
We iteratively update $G$ and $S$ by contracting edges in $G$ or by augmenting $S$ in $G$, while maintaining the following properties throughout: $G[S]$ has a spanning subgraph isomorphic to $H[\omega_{[|S|]}]$, $G\setminus S$ is nonempty, and every vertex in $G\setminus S$ has degree at least $t$. 
We apply further minor operations in $G$, if necessary, so that no vertex in $G\setminus S$ is separated from $S$ by a vertex cut of size at most $|S|$. This ensures that every vertex in $S$ has at least 2 neighbors in $G\setminus S$ and that, for every vertex $v$ in $G\setminus S$, there are $|S|$ paths from $v$ to $S$ that pairwise intersect only at $v$. The latter implies that every vertex in $G\setminus S$ can serve as the apex to form an $H[\omega_{[|S|]}]^+$-minor in $G$.

The main task is to find a suitable ordering $\omega=(\omega_1,\dots,\omega_t)$ of $H$ so that we can augment $S$ and the $H[\omega_{[|S|]}]$-minor in $G[S]$ in accordance with the ordering $\omega$ until we obtain $H=H[\omega_{[t]}]$ in $G[S]$, at which point we can use any vertex in $G\setminus S$ as the apex to obtain an $H^+$-minor in $G$.

Let us first consider the case that $H$ is a tree, which is straightforward. Since trees are 1-degenerate, $H$ admits an ordering $\omega=(\omega_1,\dots,\omega_t)$ such that for all $s\in[t-1]$, $\omega_{s+1}$ is adjacent to exactly one vertex in $\omega_{[s]}$.
Hence, if $|S|<t$, then there is a vertex $v\in V(G)\setminus S$ that is adjacent to the vertex of $S$ corresponding to $\omega_i$, where $\omega_i$ is the unique neighbor of $\omega_{|S|+1}$ in $H[\omega_{[|S|+1]}]$. Then $G[S\cup\{v\}]$ contains a spanning subgraph isomorphic to $H[\omega_{[|S|+1]}]$. In this manner, we can always augment $S$ by adding $v$.

Now suppose that $H$ is a maximal outerplanar graph. Since outerplanar graphs are 2-degenerate and $H$ is maximally outerplanar, $H$ admits an ordering $\omega=(\omega_1,\dots,\omega_t)$ such that for all $s\in[t-1]$ with $s\geq 2$, $\omega_{s+1}$ is adjacent to exactly two vertices in $\omega_{[s]}$, which are themselves adjacent in $H$.
If $|S|<t$ and if there is a vertex $v\in V(G)\setminus S$ that is adjacent to the two vertices of $S$ corresponding to $\omega_i$ and $\omega_j$, where $\omega_i,\omega_j$ are the two neighbors of $\omega_{|S|+1}$ in $H[\omega_{[|S|+1]}]$, then $G[S\cup\{v\}]$ contains a spanning subgraph isomorphic to $H[\omega_{[|S|+1]}]$ and we can augment $S$ by adding $v$ as before.

If such a vertex $v\in V(G)\setminus S$ does not exist, then we cannot augment $S$ at this stage; instead, we contract the edge in $G[S]$ corresponding to the edge $\omega_i\omega_j$ in $H[\omega_{[|S|]}]$ to obtain $G'$ and $S'$.
Since $|S'|=|S|-1$, this shrinks the set $S$, but it also shrinks the graph $G$, which is progress. Also note that the degree of every vertex in $G'\setminus S'$ is equal to that in $G\setminus S$ since we assumed that no vertex in $G\setminus S$ is adjacent to both endpoints of the contracted edge.
However, $G'[S']$ may no longer have a spanning subgraph isomorphic to $H[\omega_{[|S'|]}]$. Indeed, $H[\omega_{[|S|]}]/\omega_i\omega_j$ may not even be a subgraph of $H$; in this case, our proof fails. Even if $H[\omega_{[|S|]}]/\omega_i\omega_j$ is a subgraph of $H$, it may not be isomorphic to $H[\omega_{[|S|-1]}]$; but in this case, there may be a different ordering $\omega'$ of $H$ such that $H[\omega_{[|S|]}]/\omega_i\omega_j$ is isomorphic to $H[\omega'_{[|S|-1]}]$, and we can continue by trying to augment $S'$ in accordance with the new ordering $\omega'$. Our definition of contractible orderings encodes exactly this property, that whenever we may need to contract an edge $\omega_i\omega_j$, there is always a (possibly different) ordering $\omega'$ of $H$ such that $H[\omega_{[|S|]}]/\omega_i\omega_j\cong H[\omega'_{[|S|-1]}]$.

Let us now be precise. A \emph{recursive ordering} of a graph $H$ on $t$ vertices is an ordering $\omega=(\omega_1,\dots,\omega_t)$ of $H$ such that
    \begin{itemize}
        \item for all $s\in[t-1]$, we have that $\omega_{s+1}$ is adjacent to at most two vertices in $\omega_{[s]}$, and if $\omega_{s+1}$ is adjacent to exactly two vertices in $\omega_{[s]}$, then these two vertices are adjacent in $H$.
    \end{itemize}
A \emph{contractible ordering} of $H$ is a nonempty set $\Omega$ of recursive orderings of $H$ such that 
    \begin{itemize}
        \item for every $\omega\in\Omega$ and $s\in[t-1]$, if $\omega_{s+1}$ is adjacent to two vertices in $\omega_{[s]}$, say $\omega_i$ and $\omega_j$, then there exists $\omega'\in\Omega$ such that $H[\omega_{[s]}]/\omega_i\omega_j$ is isomorphic to $H[\omega'_{[s-1]}]$.
    \end{itemize}
If $H$ admits a contractible ordering, then we say that $H$ is \emph{contractibly orderable}.

\subsection{Proof of Theorem \ref{thm:mainmindeg}} \label{sec:proofcontractibleordering}

Let $G$ be a graph.
A \emph{separation} of $G$ is an ordered pair $(A,B)$ of subsets of $V(G)$ with $A\cup B = V(G)$ such that there does not exist an edge in $G$ with one end in $A\setminus B$ and the other end in $B\setminus A$; $(A,B)$ is \emph{nontrivial} if $A\setminus B$ and $B\setminus A$ are both nonempty, and its \emph{order} is defined to be $|A\cap B|$.
We say that a subset $S$ of $V(G)$ is \emph{well-connected in $G$} if $S\neq V(G)$ and there does not exist a nontrivial separation $(A,B)$ of $G$ with $S\subseteq A$ and $|A\cap B|<|S|$. 
Note that if $S$ is well-connected in $G$, then for all $v\in V(G)\setminus S$, there exist $|S|$ paths $P_1,\dots,P_{|S|}$ from $v$ to $S$ only sharing $v$, by Menger's theorem.
Also note that $\emptyset$ is well-connected in any nonempty graph.

We are ready to prove Theorem \ref{thm:mainmindeg}, that if $H$ is a graph on $t$ vertices that is contractibly orderable, then every graph with minimum degree $t$ contains an $H^+$-minor. For inductive purposes we first prove the following more technical statement.
\begin{lemma} \label{lem:contractableorder}
Let $t$ be a positive integer and let $H$ be a graph on $t$ vertices admitting a contractible ordering $\Omega$. 
If $G$ is a graph and $S \subseteq V(G)$ is well-connected in $G$ such that every vertex in $G\setminus S$ has degree at least $t$ in $G$ and $G[S]$ contains a spanning subgraph isomorphic to $H[\omega_{[|S|]}]$ for some $\omega\in\Omega$, then $G$ contains an $H^+$-minor such that the branch set of the apex is a singleton in $V(G)\setminus S$. 
\end{lemma}

\begin{proof}
Suppose to the contrary that there exists a counterexample formed by $G$ and $S$.
That is, $S$ is well-connected in $G$, every vertex in $G\setminus S$ has degree at least $t$ in $G$, $G[S]$ contains a spanning subgraph isomorphic to $H[\omega_{[|S|]}]$ for some $\omega\in\Omega$, and $G$ does not contain an $H^+$-minor such that the branch set of the apex (the vertex in $V(H^+)\setminus V(H)$) is a singleton $V(G)\setminus S$.

We choose $G$ and $S$ as above such that $|V(G)\setminus S|$ is minimum and, subject to this, $|S|$ is minimum. 
Let $s=|S|$.
By assumption, there exists $\omega=(\omega_1,\dots,\omega_s)\in \Omega$ such that $G[S]$ contains a spanning subgraph isomorphic to $H[\omega_{[s]}]$. Fix such an $\omega$ and an isomorphism $\phi$ from $H[\omega_{[s]}]$ to such a spanning subgraph of $G[S]$.

We first assume $s=t$.
Then $\omega_{[s]} = V(H)$ and $H[\omega_{[s]}] = H$. 
Since $S$ is well-connected in $G$, we have $V(G)\setminus S \neq\emptyset$ and, for all $v\in V(G)\setminus S$, there exist $|S|=t$ paths $P_1,\dots,P_t$ in $G$ from $v$ to $S$ only sharing $v$. 
Assuming without loss of generality that $P_i$ contains $\phi(\omega_i)$, we obtain an $H^+$-minor in $G$ such that the branch set of the apex is a singleton in $V(G)\setminus S$; namely, $\mu(a)=\{v\}$, where $a$ is the apex, and for each $i\in[t]$, $\mu(\omega_i) = V(P_i)\setminus \{v\}$. 
This contradicts our choice of $G$ and $S$.

So we may assume $s<t$.  

\begin{claim}\label{lem:contractibleorder:claim1}
There does not exist a nontrivial separation $(A,B)$ in $G$ with $S\subseteq A$ and $|A \cap B| \leq s$.
\end{claim}

\begin{subproof}[Proof of Claim \ref{lem:contractibleorder:claim1}]
Suppose to the contrary that such a separation $(A,B)$ exists. 
Since $S$ is well-connected in $G$, there does not exist a nontrivial separation $(A,B)$ in $G$ with $S\subseteq A$ and $|A \cap B|<s$, so we have $|A \cap B|=s$.
Since $S$ is well-connected in $G$, there exist $s$ disjoint paths $P_1,\dots,P_s$ in $G$ from $S$ to $A\cap B$. Let us assume without loss of generality that $\phi(\omega_i)\in V(P_i)$ for $i\in[s]$. 
Let $G'$ be the graph obtained from $G$ by contracting $P_i$ into a single vertex $v_i$ for each $i\in[s]$ and deleting every vertex in $A\setminus (\bigcup_{i\in[s]} V(P_i))$. Let $S'=\{v_1,\dots,v_s\}$. 
Then $S'$ is well-connected in $G'$ and each vertex in $G'\setminus S'$ has degree in $G'$ equal to that in $G$, which is at least $t$.
Moreover, $G'[S']$ contains a spanning subgraph isomorphic to $H[\omega_{[s]}]$, witnessed by the map $\omega_i \mapsto v_i$ for $i\in[s]$. 
On the other hand, since $(A,B)$ is a nontrivial separation, we have $S\subsetneq A$, so $|V(G')\setminus S'|<|V(G)\setminus S|$.
By minimality of $|V(G)\setminus S|$, $G'$ contains an $H^+$-minor $\mu'$ such that the branch set $\mu'(a)$ of the apex $a$ is a singleton in $V(G')\setminus S'$.
Now let $\mu(a)=\mu'(a)$ and, for each $i\in[s]$, let $\mu(\omega_i) = (\mu'(\omega_i)\setminus\{v_i\}) \cup V(P_i)$.
Then $\mu$ is an $H^+$-minor in $G$ such that the branch set of the apex is a singleton in $V(G)\setminus S$, a contradiction.
\end{subproof}

Since $\omega$ is a recursive ordering of $H$, $\omega_{s+1}$ is adjacent to at most two vertices in $\omega_{[s]}$ in $H$, and if $\omega_{s+1}$ is adjacent to exactly two vertices in $\omega_{[s]}$, then these two vertices are adjacent in $H$.

\begin{claim}\label{lem:contractibleorder:claim2}
If $\omega_{s+1}$ is adjacent to exactly two vertices in $\omega_{[s]}$, say $\omega_i$ and $\omega_j$, then there exists a vertex in $V(G)\setminus S$ adjacent to both $\phi(\omega_i)$ and $\phi(\omega_j)$.
\end{claim}

\begin{subproof}[Proof of Claim \ref{lem:contractibleorder:claim2}]
Suppose to the contrary that no vertex in $V(G)\setminus S$ is adjacent to both $\phi(\omega_i)$ and $\phi(\omega_j)$.
Since $\omega$ is a recursive ordering of $H$, we have $\omega_i\omega_j\in E(H)$ and, hence, $\phi(\omega_i)\phi(\omega_j)\in E(G)$.
Moreover, since $\Omega$ is a contractible ordering of $H$, there exists $\omega'\in \Omega$ such that $H[\omega_{[s]}]/\omega_i\omega_j$ is isomorphic to $H[\omega'_{[s-1]}]$. 

Let $G'$ be the graph obtained from $G$ by contracting the edge $\phi(\omega_i)\phi(\omega_j)$ into a vertex $v$.
Let $S' = (S\setminus \{\phi(\omega_i),\phi(\omega_j)\})\cup\{v\}$.
Since $S$ is well-connected in $G$, $S'$ is well-connected in $G'$. 
For every $z \in V(G) \setminus S = V(G')\setminus S'$, since $z$ has degree at least $t$ in $G$ and is not adjacent to both $\phi(\omega_i)$ and $\phi(\omega_j)$, $z$ still has degree at least $t$ in $G'$.
Moreover, since $G'[S'] = G[S]/\phi(\omega_i)\phi(\omega_j)$ and since $H[\omega_{[s]}]/\omega_i\omega_j$ is isomorphic to $H[\omega'_{[s-1]}]$, we have that $G'[S']$ contains a spanning subgraph isomorphic to $H[\omega'_{[s-1]}]$.
On the other hand, we have $|V(G)\setminus S|=|V(G')\setminus S'|$ and $|S'|<|S|$. It follows from our choice of $G$ and $S$ that $G'$ has an $H^+$-minor such that the branch set of the apex is a singleton in $V(G')\setminus S'$. Uncontracting the edge $\phi(\omega_i)\phi(\omega_j)$ gives an $H^+$-minor in $G$ such that the branch set of the apex is a singleton in $V(G)\setminus S$, a contradiction.
\end{subproof}

We now add a vertex $v \in V(G) \setminus S$ to $S$, where $v$ is chosen as follows:
    \begin{description}
        \item[Case 1:] If $\omega_{s+1}$ is not adjacent to any vertex in $\omega_{[s]}$, then let $v$ be an arbitrary vertex in $V(G)\setminus S$.
        \item[Case 2:] If $\omega_{s+1}$ is adjacent to exactly one vertex in $\omega_{[s]}$, say $\omega_i$, then let $v$ be a vertex in $V(G)\setminus S$ adjacent to $\phi(\omega_i)$. 
        
        (Note that such a vertex $v$ exists, for otherwise $(S, V(G) \setminus \{\phi(\omega_i)\})$ is a separation of $G$ with $|S \cap V(G) \setminus \{\phi(\omega_i)\}| \leq s-1$, contradicting the well-connectedness of $S$.) 
        \item[Case 3:] If $\omega_{s+1}$ is adjacent to exactly two vertices in $\omega_{[s]}$, say $\omega_i$ and $\omega_j$, then let $v$ be a vertex in $V(G)\setminus S$ adjacent to both $\phi(\omega_i)$ and $\phi(\omega_j)$. 
        
        (Note that the existence of $v$ follows from Claim \ref{lem:contractibleorder:claim2}.)
    \end{description}
In all cases, define $S'=S\cup\{v\}$ and note that $G[S']$ contains a spanning subgraph isomorphic to $H[\omega_{[s+1]}]$, witnessed by the map $\omega_i\mapsto \phi(\omega_i)$ for $i\in[s]$ and $\omega_{s+1}\mapsto v$. Moreover, $S'$ is well-connected in $G$ by Claim \ref{lem:contractibleorder:claim1}. But since $|V(G)\setminus S'| < |V(G)\setminus S|$, we have by minimality of $|V(G)\setminus S|$ that $G$ contains an $H^+$-minor such that the branch set of the apex is a singleton in $V(G)\setminus S' \subseteq V(G)\setminus S$, a contradiction.
\end{proof}

Theorem \ref{thm:mainmindeg}, restated below, now follows immediately from Lemma \ref{lem:contractableorder}.
\mainmindeg*
\begin{proof}
Let $t=|V(H)|$.
By assumption, there exists a graph $H_1$ on $t$ vertices admitting a contractible ordering $\Omega$ such that $H$ is a spanning subgraph of $H_1$.
Note that $\emptyset$ is well-connected in $G$, every vertex in $G \setminus \emptyset$ has degree at least $t$, and $G[\emptyset]$ is isomorphic to $H[\omega_{[|\emptyset|]}]=H[\omega_\emptyset]$ for any $\omega \in \Omega$, since these are both the empty graph.
By Lemma \ref{lem:contractableorder}, there exists an $H_1^+$-minor in $G$ such that the branch set of the apex is a singleton.
Since $H$ is a subgraph of $H_1$, this gives an $H^+$-minor in $G$ such that the branch set of the apex is a singleton.
\end{proof}

\subsection{Examples of contractibly orderable graphs} \label{sec:COexamples}
In this section we show that several classes of outerplanar graphs are contractibly orderable. 
We first define a stronger notion of \emph{rooted} contractible orderings and prove in Subsection \ref{sec:universal} that \emph{universal} outerplanar graphs (maximal outerplanar graphs whose weak dual tree is a complete cubic tree) are contractibly orderable. 
The case where the weak dual tree is a complete binary tree is similar.
In Subsection \ref{sec:moreCOexamples}, we describe a general method to construct larger contractibly orderable graphs using rooted contractible orderings, and prove that cacti and maximal outerplanar graphs whose weak dual tree is a path are contractibly orderable.

We begin by defining rooted contractible orderings.
Let $H$ be a graph, and let $\rho=(\rho_1,\dots,\rho_m)$ be a sequence of vertices of $H$ for some integer $m$ with $0 \leq m \leq 2$.
Let $t=|V(H)|$.
We say that an ordering $\omega=(\omega_1,\dots,\omega_t)$ of $H$ is \emph{$\rho$-rooted} if $(\omega_1,...,\omega_m)=\rho$. 
A \emph{$\rho$-rooted contractible ordering} of $H$ is a nonempty set $\Omega$ of recursive orderings of $H$ such that 
    \begin{itemize}
        \item every $\omega\in\Omega$ is $\rho$-rooted, 
        
        \item for every $\omega\in\Omega$ and $s\in[t-1]$ with $s\geq 3$, if $\omega_{s+1}$ is adjacent to two vertices in $\omega_{[s]}$, say $\omega_i$ and $\omega_j$, then $\{\omega_i,\omega_j\}\neq\{\rho_1,\dots,\rho_m\}$ and there exists $\omega'\in\Omega$ and an isomorphism $\phi$ from $H[\omega'_{[s-1]}]$ to $H[\omega_{[s]}]/\omega_i\omega_j$ such that $\phi(\rho_k)=\rho_k$ for every $k\in[m]$.

        (Note that $\omega_i\omega_j\in E(H)$ because every $\omega\in\Omega$ is a recursive ordering. 
        Moreover, if $m=2$, then $\rho_1,\rho_2$ are distinct elements in $V(H[\omega_{[s-1]}'])$ and in $V(H[\omega_{[s]}]/\omega_i\omega_j)$ because $s\geq 3$ and $\{\omega_i,\omega_j\}\neq\{\rho_1,\dots,\rho_m\}$ respectively.)
    
    \end{itemize}
    Note that the second condition implies that
    if $m=2$ and $\rho_1\rho_2$ is an edge of $H$ belonging to a triangle of $H$, and if $v\in V(H)$ is adjacent to both $\rho_1$ and $\rho_2$, then $v=\omega_3$ for all $\omega=(\omega_1,\dots,\omega_t)\in\Omega$.
    Also note that every $\rho$-rooted contractible ordering $\Omega$ of $H$ is also a contractible ordering of $H$: Let $\omega\in\Omega$, $s\in[t-1]$, and suppose $\omega_{s+1}$ is adjacent to two vertices $\omega_i$ and $\omega_j$ in $\omega_{[s]}$. Note that this implies $s\geq 2$. If $s=2$, then $H[\omega_{[s]}]\cong K_2$, so $H[\omega_{[s]}]/\omega_i\omega_j \cong K_1\cong H[\omega_{[s-1]}]$. If $s\geq 3$, then there exists $\omega'\in\Omega$ and an isomorphism from $H[\omega'_{[s-1]}]$ to $H[\omega_{[s]}]/\omega_i\omega_j$ by the definition of a $\rho$-rooted contractible ordering.

To demonstrate a very simple example of a rooted contractible ordering, we show that every tree has such an ordering, rooted at any given vertex.
We remark that this result can also be deduced from a more general result proved later in this paper (Lemma \ref{lem:attachingroots}).

\begin{proposition} \label{prop:tree}
Let $H$ be a tree and let $v\in V(H)$. 
Then $H$ admits a $(v)$-rooted contractible ordering.
\end{proposition}

\begin{proof}
Let $t=|V(H)|$ and let $\omega=(\omega_1,\dots,\omega_t)$ be a recursive ordering of $H$ with $\omega_1=v$ (for example, a breadth first search ordering rooted at $v$). 
Then $\omega$ is clearly a $(v)$-rooted recursive ordering of $H$.
Since for every $s\in[t-1]$, $\omega_{s+1}$ is adjacent to exactly one vertex in $\omega_{[s]}$, it follows that $\{\omega\}$ is a $(v)$-rooted contractible ordering of $H$.
\end{proof}

\subsubsection{Universal outerplanar graphs}\label{sec:universal}

Let $H$ be a maximal outerplanar graph on at least three vertices. 
Note that $H$ has a unique outerplanar embedding (that is, an embedding in the plane such that every vertex is incident to the unbounded region), up to homeomorphisms of the plane. 
The \emph{weak dual tree} of $H$ is the tree $T$ that has a vertex for each bounded face of the outerplanar embedding of $H$ and an edge joining pairs of vertices of $T$ whose corresponding faces share an edge in the outerplanar embedding of $H$.
If $H$ is a maximal outerplanar graph with at most two vertices, then we define the weak dual tree of $H$ to be the empty graph.

If $H$ is a maximal outerplanar graph and if $\omega=(\omega_1,\dots,\omega_t)$ is a recursive ordering of $H$, then $\omega_{s+1}$ is adjacent to exactly two vertices in $\omega_{[s]}$ for all $s\in[t-1]$ with $s\geq 2$.
Note that each recursive ordering $\omega=(\omega_1,\dots,\omega_t)$ of $H$ gives a recursive ordering $\tau=(\tau_1,\dots,\tau_{t-2})$ of the weak dual tree $T$ of $H$ (and vice versa), where for each $s\in[t-1]$ with $s\geq 2$, $\tau_{s-1}$ is the bounded face of $H$ whose three incident vertices are $\omega_{s+1}$ and the two vertices in $\omega_{[s]}$ adjacent to $\omega_{s+1}$.

Let us collect some more easy observations.
\begin{lemma}\label{lem:weakdualtreeobs}
    Let $H$ be a maximal outerplanar graph on at least three vertices and let $T$ be the weak dual tree of $H$.
    Then the following statements hold.
    \begin{enumerate}[label={(\alph*)}]
        \item For each subtree $T'$ of $T$, there is a maximal outerplanar graph $H'$ that is an induced subgraph of $H$ such that $T'$ is the weak dual tree of $H'$; for every maximal outerplanar induced subgraph $H'$ of $H$, the weak dual tree of $H'$ is a subtree of $T$. 
        \label{lem:weakdualtreeobs1}
        \item If $e$ is an edge on the boundary of the outerplanar embedding of $H$, then the weak dual tree of $H/e$ is obtained from $T$ by contracting an edge incident to $f$, where $f$ is the bounded face of $H$ incident to $e$. \label{lem:weakdualtreeobs2}
    \end{enumerate}
\end{lemma}
\begin{proof}
Let $\ell$ be a vertex of degree at most 1 in $T$. Then $\ell$ is a bounded face in the outerplanar embedding of $H$ with at least two of its incident edges on the boundary; let $v_\ell$ be a vertex of degree 2 in $H$ incident with $\ell$. Then $T\setminus\ell$ is the weak dual tree of $H\setminus v_\ell$.

Every subtree $T'$ of $T$ can be obtained from $T$ by repeatedly deleting vertices of degree at most 1, so $T'$ is the weak dual tree of an induced subgraph of $H$ obtained by repeatedly deleting the corresponding vertices of degree 2. 
Conversely, every maximal outerplanar graph $H'$ that is an induced subgraph of $H$ can be obtained from $H$ by repeatedly deleting vertices of degree 2, and the weak dual tree of $H'$ can be obtained from $T$ by repeatedly deleting the corresponding vertices of degree at most 1.
This proves the first statement of the lemma.

Now let $e$ be an edge on the boundary of $H$. Since $H$ is a maximal outerplanar graph, $e$ is incident to a bounded face $f$ which has length 3. Let $z$ denote the vertex incident to $f$ but not to $e$. Then $H/e$ is obtained by contracting $e$ and deleting one of the two parallel edges joining $z$ to the contracted vertex. In the planar dual, this corresponds to deleting the edge joining $f$ to the unbounded face of $H$, and contracting one of the two remaining edges incident with $f$. Hence the weak dual tree of $H/e$ is obtained from $T$ by contracting an edge incident to $f$. This proves the second statement of the lemma.
\end{proof}

If $H$ is a maximal outerplanar graph, then every vertex of the weak dual tree of $H$ has degree at most three.
A {\it cubic tree} is a 1-vertex tree or a tree whose every vertex has degree one or three.
For a nonnegative integer $h$, the {\it complete cubic tree with height $h$} is a cubic tree that contains a vertex $r$ such that every path from $r$ to a vertex of degree one contains exactly $h$ edges.
The {\it universal outerplanar graph with height $h$} is the maximal outerplanar graph $H$ such that its weak dual tree is the complete cubic tree with height $h$.

An \emph{$h$-center} of a graph $T$ is a vertex $r\in V(T)$ such that for every $v\in V(T)$, there is a path in $T$ from $r$ to $v$ with at most $h$ edges.
The \emph{radius} of a graph $T$ is the smallest integer $h$ for which an $h$-center exists. 
Note that if $T$ is a complete cubic tree with height $h$, then the radius of $T$ is $h$ and $T$ has a unique $h$-center.

The following result justifies the name of universal outerplanar graphs. 

\begin{lemma} \label{lem:universalouterplanar}
Let $h$ be a nonnegative integer.
Let $G$ be the universal outerplanar graph with height $h$.
Let $T$ be the weak dual tree of $G$.
Let $r$ be the $h$-center of $T$ and let $x^1,x^2,x^3$ be the three vertices of $G$ incident to $r$.
Let $H$ be a maximal outerplanar graph on at least three vertices such that its weak dual tree has an $h$-center $r_H$, and let $x_H^1, x_H^2,x_H^3$ be the three vertices of $H$ incident to $r_H$.
Then there exists a subgraph embedding\footnote{A {\it subgraph embedding} from a graph $H$ to a graph $G$ is an injection $\phi: V(H) \rightarrow V(G)$ such that $\phi(x)\phi(y) \in E(G)$ for every $xy \in E(H)$.} $\phi$ from $H$ to $G$ such that $\phi(x_H^i)=x^i$ for all $i\in[3]$.
\end{lemma}

\begin{proof}
We prove this lemma by induction on $h$.
When $h=0$, both $G$ and $H$ are isomorphic to $K_3$, so the desired subgraph embedding $\phi$ clearly exists.
So we may assume that $h \geq 1$ and that the lemma holds for $h-1$.

Since $G$ is the universal outerplanar graph with height $h$, $T$ is the complete cubic tree with height $h$, and $r$ is the unique $h$-center of $T$.
Let $T_1$ be the subgraph of $T$ such that $T_1$ is the complete cubic tree with height $h-1$ centered at $r$.
By Lemma \ref{lem:weakdualtreeobs}\ref{lem:weakdualtreeobs1}, $T_1$ is the weak dual tree of a subgraph $G_1$ of $G$, and $G_1$ is a universal outerplanar graph with height $h-1$.
Moreover, there exists a bijection $\iota$ from $V(G) \setminus V(G_1)$ to the set of edges of $G_1$ incident to a vertex of $G_1$ with degree two in $G_1$ such that for every $v \in V(G) \setminus V(G_1)$, $v$ is adjacent to both ends of $\iota(v)$, and $\iota(v)$ is incident to a face of $G_1$ corresponding to a vertex of $T$ with distance exactly $h-1$ from $r$.

Let $T_H$ be the weak dual tree of $H$.
By assumption, $r_H$ is an $h$-center of $T_H$.
Let $T_H'$ be a subgraph of $T_H$ induced by the vertices of $T_H$ with distance from $r_H$ at most $h-1$, so that $r_H$ is an $(h-1)$-center of $T_H'$.
By Lemma \ref{lem:weakdualtreeobs}\ref{lem:weakdualtreeobs1}, $T_H'$ is the weak dual tree of a maximal outerplanar induced subgraph $H_1$ of $H$.
Moreover, there exists an injection $\iota_H$ from $V(H) \setminus V(H_1)$ to the set of edges of $H_1$ incident to a vertex of $H_1$ with degree two in $H_1$ such that for every $v \in V(H) \setminus V(H_1)$, $v$ is adjacent to both ends of $\iota_H(v)$, and $\iota_H(v)$ is incident to a face of $H_1$ corresponding to a vertex of $T_H$ with distance exactly $h-1$ from $r_H$.

By the inductive hypothesis, there exists a subgraph embedding $\phi$ from $H_1$ to $G_1$ such that $\phi(x_H^i)=x^i$ for all $i\in[3]$.
Then for every $v \in V(H) \setminus V(H_1)$, there exists $v' \in V(G) \setminus V(G_1)$ such that $\iota_H(v)=\iota(v')$, and we define $\phi(v)=v'$.
Note that $\phi$ is a subgraph embedding from $H$ to $G$ such that $\phi(x_H^i)=x^i$ for all $i\in[3]$.
\end{proof}

We now prove that universal outerplanar graphs are contractibly orderable. In fact, they admit an $(x)$-rooted contractible ordering, where $x$ is a vertex incident to the center of its weak dual tree.

\begin{proposition}\label{prop:universal}
For every nonnegative integer $h$, the universal outerplanar graph $H$ with height $h$ admits an $(x)$-rooted contractible ordering, where $x$ is a vertex incident to the $h$-center of the weak dual tree of $H$.
\end{proposition}

\begin{proof}
Let $h$ be a nonnegative integer, and let $H$ be the universal outerplanar graph with height $h$.
Let $T$ be the weak dual tree of $H$. 
Let $f$ be the $h$-center of $T$ and let $x,y,z$ denote the three vertices of $H$ incident to the face $f$.
Let $\Omega$ be the set of recursive orderings of $H$ such that the first three entries are $(x,y,z)$.
Obviously, $\Omega \neq \emptyset$, and every $\omega\in\Omega$ is $(x)$-rooted.

Let $\omega \in \Omega$ and let $s \in [|V(H)|-1]$ with $s\geq 3$.
Since $H$ is a maximal outerplanar graph, $\omega_{s+1}$ is adjacent to two vertices in $\omega_{[s]}$, say $\omega_i$ and $\omega_j$, and since $\omega$ is a recursive ordering of $H$, we have $\omega_i\omega_j\in E(H)$.
Clearly $\{\omega_i,\omega_j\}\neq \{x\}$, so it suffices to show that  and that there exists $\omega'\in \Omega$  and an isomorphism $\phi$ from $H[\omega_{[s]}]/\omega_i\omega_j$  to $H[\omega'_{[s-1]}]$  such that $\phi(x)=x$.

Note that for every maximal outerplanar induced subgraph $H'$ of $H$ containing the three vertices incident to $f$, there exists $\omega'\in\Omega$ such that $H' = H[\omega'_{[|V(H')|]}]$.

Since $\omega$ is a recursive ordering of $H$, $H[\omega_{[s]}]$ is a maximal outerplanar subgraph of $H$ and $\omega_i\omega_j$ is an edge on the boundary of the outerplanar embedding of $H[\omega_{[s]}]$.
Hence $H[\omega_{[s]}]/\omega_i\omega_j$ is a maximal outerplanar graph and the weak dual tree $T'$ of $H[\omega_{[s]}]/\omega_i\omega_j$ is a minor of $T$ by Lemma \ref{lem:weakdualtreeobs}, so $f$ is still an $h$-center of $T'$.
By Lemma \ref{lem:universalouterplanar}, there exists a subgraph embedding $\phi$ from $H[\omega_{[s]}]/\omega_i\omega_j$ to $H$ such that $\phi(x)=x$.
Since $\phi$ is a subgraph embedding, it is an isomophism from 
$H[\omega_{[s]}]/\omega_i\omega_j$ to its image, which is a maximal outerplanar subgraph of $H$.
By the note in the previous paragraph, there exists $\omega'\in\Omega$ such that $H[\omega'_{[s-1]}]$ is the image of $\phi$.
In other words, $\phi$ is an isomorphism from $H[\omega_{[s]}]/\omega_i\omega_j$ to $H[\omega'_{[s-1]}]$ such that $\phi(x)=x$.
\end{proof}

A similar argument applies to maximal outerplanar graphs whose weak dual trees are complete binary trees: for a positive integer $h$, a \emph{complete binary tree with height $h$} is a tree $T$ with a root vertex $r$ obtained from a complete cubic tree $T'$ of height $h$ with $h$-center $r$ by deleting $C$, where $C$ is a connected component of $T'\setminus r$.
If $H$ is a maximal outerplanar graph whose weak dual tree is a complete binary tree $T$ with root $r$, then there is an edge $xy$ on the boundary of $H$ incident to the face $r$.
A straightforward adaptation of the proof of Proposition \ref{prop:universal} shows that the set $\Omega$ of all $(x,y)$-rooted recursive orderings of $H$ is an $(x,y)$-rooted contractible ordering; we omit this proof. 
\begin{proposition}
    \label{prop:binarytree}
    For every positive integer $h$, if $H$ is a maximal outerplanar graph whose weak dual tree is a complete binary tree with height $h$ and root $r$, then $H$ admits an $(x,y)$-rooted contractible ordering, where $xy$ is an edge on the boundary of $H$ incident to $r$.
\end{proposition}

\subsubsection{Constructing more contractibly orderable graphs} \label{sec:moreCOexamples}

We now give a general construction to obtain larger graphs with rooted contractible orderings from smaller ones.

Let $H$ be a graph on $t$ vertices for some positive integer $t$.
Let $\rho=(\rho_1,\dots,\rho_m)$ be a sequence of vertices of $H$ for some nonnegative integer $m$ with $0 \leq m \leq 2$.
We say that an edge $e$ of $H$ is \emph{$\rho$-contractible} if
    \begin{itemize}
        \item there exist a $\rho$-rooted contractible ordering $\Omega$, $\omega \in \Omega$, and an isomorphism $\phi$ from $H[\omega_{[t-1]}]$ to $H/e$ such that $\phi(\rho_i)=\rho_i$ for every $i \in [\min\{m,t-1\}]$.
    \end{itemize}
In this case, we call $\Omega$ a {\it witness} for $e$.

Note that if $e$ is $\rho$-contractible and $H$ has at least two edges, then at most one end of $e$ is an entry of $\rho$.
If $H$ admits a $\rho$-contractible ordering $\Omega$, then for each $\omega= (\omega_1,\dots,\omega_t)\in\Omega$ such that $\omega_t$ has degree at least 1 in $H$, we have that each edge $e$ incident to $\omega_t$ is $\rho$-contractible because $H/e$ is equivalent to $H[\omega_{[t-1]}]$.
In particular, if $H$ is a connected graph with at least one edge that admits a $\rho$-rooted contractible ordering, then $H$ has a $\rho$-contractible edge.
Also note that, by definition, if there exists a $\rho$-contractible edge of $H$, then $H$ admits a $\rho$-rooted contractible ordering.

The next lemma describes three ways to combine two contractibly orderable graphs to form a larger contractibly orderable graph:
If $H$ and $H'$ are contractibly orderable, then their disjoint union is contractibly orderable; if $H'$ admits a $(v)$-rooted contractible ordering, then gluing $H'$ to $H$ by identifying $v$ with any vertex of $H$ gives a contractibly orderable graph; finally, if $H'$ admits a $(u,v)$-rooted contractible ordering, then gluing $H'$ to $H$ by identifying $(u,v)$ with a contractible edge of $H$ gives a contractibly orderable graph.

\begin{lemma} \label{lem:attachingroots}
Let $H$ and $H'$ be graphs.
Let $\rho=(\rho_1,\dots,\rho_m)$ and $\rho'=(\rho'_1,\dots,\rho'_{m'})$ be sequences of vertices of $H$ and $H'$, respectively, for some $m$ and $m'$ with $m,m' \in \{0,1,2\}$.
Suppose $H$ admits a $\rho$-rooted contractible ordering $\Omega$, and suppose $H'$ admits a $\rho'$-rooted contractible ordering.
If one of the following conditions hold:
    \begin{enumerate}
        \item $m'=0$ and $V(H)\cap V(H')=\emptyset$,
        \item $m'=1$ and $V(H)\cap V(H')=\{\rho'_1\}$,
        \item $m'=2$, $V(H)\cap V(H')=\{\rho'_1,\rho'_2\} \neq \{\rho_1,\rho_2\}$,
        and $\rho'_1\rho'_2$  a $\rho$-contractible edge of $H$ with $\Omega$ a witness,
    \end{enumerate}
then $H\cup H'$ admits a $\rho$-rooted contractible ordering.
Moreover, if $e'$ is a $\rho'$-contractible edge of $H'$, then $e'$ is a $\rho$-contractible edge of $H \cup H'$.
\end{lemma}

\begin{proof}
Let $t=|V(H)|$ and $t'=|V(H')|$.
Let $G=H \cup H'$ and $\tau=|V(G)|=t+t'-m'$.
Let $\Omega'$ be a $\rho'$-rooted contractible ordering of $H'$.
If $H'$ has a $\rho'$-contractible edge, then fix such an edge $e'$ and assume without loss of generality that $\Omega'$ is a witness for $e'$.

For each $\omega=(\omega_1,\dots,\omega_t)\in\Omega$ and $\omega'=(\omega'_1,\dots,\omega'_{t'})\in\Omega'$, define $\psi_{\omega,\omega'} = (\omega_1,\dots,\omega_t,\omega'_{m'+1},\dots,\omega'_{t'})$; clearly, $\psi_{\omega,\omega'}$ is a $\rho$-rooted ordering of $G$.
Let 
    \begin{equation*}
        \Psi=\{\psi_{\omega,\omega'}: \omega\in\Omega \text{ and } \omega'\in\Omega'\}.
    \end{equation*}
Since $\Omega'$ is a $\rho'$-rooted contractible ordering of $H'$, $\{\omega'_i: i \in [m']\}=\{\rho'_i: i \in [m']\}=V(H) \cap V(H')$ for every $\omega' \in \Omega'$, so every member of $\Psi$ is a $\rho$-rooted ordering of $G$.

To prove this lemma, it suffices to show that $\Psi$ is a $\rho$-rooted contractible ordering of $G$ and that, if $e'$ is a $\rho'$-contractible edge of $H'$, then $e'$ is a $\rho$-contractible edge of $G$ with witness $\Psi$. 
Let us first show that $\Psi$ is a $\rho$-rooted contractible ordering of $G$.

    Clearly, every member of $\Psi$ is a $\rho$-rooted ordering of $G$.
    If $m=2$ and $v$ is a vertex of $G$ adjacent to both $\rho_1$ and $\rho_2$, then by the third condition of the lemma, we have $\{\rho_1,\rho_2\} \not \subseteq V(H')$, so $v \in V(H)$; since $\Omega$ is a $\rho$-contractible ordering of $H$, $v=\omega_3$ for every $\omega \in \Omega$, so $v=\psi_3$ for every $\psi \in \Psi$.

    Let $\psi \in \Psi$; then there exists $\omega \in \Omega$ and $\omega' \in \Omega'$ such that $\psi=\psi_{\omega,\omega'}$.
    It remains to show that, for all $s\in[\tau-1]$:
    \begin{enumerate}
        \item[(i)] $\psi_{s+1}$ is adjacent to at most two vertices in $\psi_{[s]}$, and
        \item[(ii)] if $\psi_{s+1}$ is adjacent to exactly two vertices in $\psi_{[s]}$, say $\psi_i$ and $\psi_j$, then 
            \begin{enumerate}
                \item $\psi_i\psi_j\in E(G)$, and
                \item if $s\geq 3$, then $\{\psi_i,\psi_j\}\neq \{\rho_1,\dots,\rho_m\}$ and there exists $\psi'\in\Psi$ and an isomorphism $\phi$ from $G[\psi'_{[s-1]}]$ to $G[\psi_{[s]}]/\psi_i\psi_j$ such that $\phi(\rho_k)=\rho_k$ for every $k\in[m]$.
            \end{enumerate} 
        \end{enumerate}

    First consider the case that $s\in[t-1]$. 
Then $\psi_{s+1}=\omega_{s+1}$ and $\psi_{[s]}=\omega_{[s]}$.
Since $\omega$ is a recursive ordering of $H$, we have that $\psi_{s+1}=\omega_{s+1}$ is adjacent to at most two vertices in $\psi_{[s]}=\omega_{[s]}$.
If $\psi_{s+1}$ is adjacent to exactly two vertices in $\psi_{[s]}$, say $\psi_i$ and $\psi_j$, then $\psi_i\psi_j=\omega_i\omega_j\in E(H)\subseteq E(G)$ since $\omega$ is a recursive ordering of $H$; furthermore, since $\Omega$ is a $\rho$-rooted contractible ordering of $H$, if $s\geq 3$, then there exists $\omega^1\in\Omega$ and an isomorphism $\phi$ from $H[\omega^1_{[s-1]}]=G[\psi^1_{[s-1]}]$ to $H[\omega_{[s]}]/\omega_i\omega_j = G[\psi^1_{[s]}]/\psi^1_i\psi^1_j$ such that $\phi(\rho_k)=\rho_k$ for all $k\in[m]$, where $\psi^1 = \psi_{\omega^1,\omega'}$.

So we may assume $s\in [\tau-1]\setminus[t-1]$.
Define $s' = s-t+m'\geq m'$, so that $\psi_{s+1} = \omega'_{s'+1}\in V(H')\setminus V(H)$. Then the neighborhood of $\psi_{s+1}$ in $\psi_{[s]}$ is contained in $\omega'_{[s']} = \psi_{[s]\setminus [t]} \cup \{\rho'_1,\dots,\rho'_{m'}\}$.
Since $\omega'$ is a recursive ordering of $H'$, it follows that $\psi_{s+1}$ is adjacent to at most two vertices in $\omega'_{[s
]}$ (hence in $\psi_{[s]}$), so (i) holds.

Therefore it suffices to show (ii).
Suppose that $\psi_{s+1}$ is adjacent to exactly two vertices in $\psi_{[s]}$, say $\psi_i$ and $\psi_j$.
Again, since $\omega'$ is a recursive ordering of $H'$, we have $\psi_i\psi_j\in E(H')\subseteq E(G)$ and (ii)(a) holds.

To show (ii)(b), assume $s\geq 3$.
First suppose $\{\psi_i,\psi_j\}\subseteq V(H)$. 
Since $\psi_{s+1}$ is a vertex in $V(H')\setminus V(H)$, we have $\{\psi_i,\psi_j\}\subseteq V(H')$, hence $\{\psi_i,\psi_j\}\subseteq V(H)\cap V(H') = \{\rho'_1,\rho'_2\}$.
This implies $m'=2$ and $\{\psi_i,\psi_j\} = \{\rho'_1,\rho'_2\}$. 
By the assumptions of the lemma, $\rho'_1\rho'_2$ is a $\rho$-contractible edge of $H$ with witness $\Omega$, so $\psi_i\psi_j=\rho_1'\rho_2' \in E(H)\cap E(H')$, and there exist $\omega^2\in \Omega$ and an isomorphism $\phi$ from $H[\omega^2_{[t-1]}]$ to $H/\psi_i\psi_j$ such that $\phi(\rho_k)=\rho_k$ for all $k\in[\min\{m,t-1\}]$.
Let $\psi^2=\psi_{\omega^2,\omega'}$.
Since $\Omega'$ is a $\rho'$-rooted contractible ordering of $H'$, and $\psi_{s+1}$ is adjacent to both ends of $\psi_i\psi_j=\rho_1'\rho_2'$, we have $\psi_{s+1} = \omega'_3=\omega'_{m'+1}$, hence $s=t$ and $\psi_{[s]}=V(H)=\psi^2_{[s]}$.
Hence $\phi$ is an isomorphism from $H[\omega^2_{[s-1]}]=G[\psi^2_{[s-1]}]$ to $H/\psi_i\psi_j=G[\psi^2_{[s]}]/\psi_i\psi_j$ such that $\phi(\rho_k)=\rho_k$ for all $k\in[m]$ (where $m = \min\{m,t-1\}$ because $m\leq 2$ and  $t=s\geq 3$).

Therefore, we may assume $\{\psi_i,\psi_j\}\not\subseteq V(H)$.
Recall that $s'=s-t+m'$, so $\psi_s=\omega'_{s'}$.
Since $\{\psi_i,\psi_j\}\subseteq \omega'_{[s']}$, we have $s'\geq 2$. If $s'=2$, then $\{\psi_i,\psi_j\}=\{\omega'_1,\omega'_2\}$ and $G[\psi_{[s]}]/\psi_i\psi_j = (H\cup H'[\omega'_{[2]}])/\omega'_1\omega'_2$.
Since $\{\psi_i,\psi_j\}\not\subseteq V(H)$, we have that $G[\psi_{[s]}]/\psi_i\psi_j$ is equivalent to $(H\cup H'[\omega'_{[2]}])\setminus \omega'_2 = H\cup H'[\{\omega'_1\}] = G[\psi_{[s-1]}]$. 
In particular, there is an isomorphism from $G[\psi_{[s-1]}]$ to $G[\psi_{[s]}]/\psi_i\psi_j$ mapping $\rho_k$ to itself for all $k\in[m]$.

We may thus assume $s'\geq 3$.
Since $\Omega'$ is a $\rho'$-rooted contractible ordering of $H'$, there exists $\omega^3\in\Omega'$ and an isomorphism $\phi'$ from $H'[\omega^3_{[s'-1]}]$ to $H'[\omega'_{[s']}]/\psi_i\psi_j$ such that $\phi'(\rho'_k)=\rho'_k$ for all $k\in [m']$.
Define $\psi^3=\psi_{\omega,\omega^3}$.
Note that $\psi^3_{[s-1]} = V(H)\cup \omega^3_{[s'-1]}$.
Let $\phi^3$ be the map from $G[\psi^3_{[s-1]}]$ to $G[\psi_{[s]}]/\psi_i\psi_j$ defined by $\phi^3(v)=v$ for all $v\in V(H)$ and $\phi^3(v)=\phi'(v)$ for all $v\in \omega^3_{[s'-1]}$.
Note that $\phi^3$ is well-defined since $V(H)\cap V(H')=\{\rho'_1,\dots,\rho'_{m'}\}$ and $\phi'(\rho'_k)=\rho'_k$ for all $k\in[m']$, and that $\phi^3$ is an isomorphism from $G[\psi^3_{[s-1]}]$ to $G[\psi_{[s]}]/\psi_i\psi_j$.
We also have $\phi^3(\rho_k)=\rho_k$ for all $k\in[m]$ because $\phi^3$ is the identity map on $V(H)$.  
This completes the proof that $\Psi$ is a $\rho$-rooted contractible ordering of $G$.

It remains to show that, if $e'$ is a $\rho'$-contractible edge of $H'$, then $e'$ is a $\rho$-contractible edge of $G$ with witness $\Psi$.
Let $e'$ be a $\rho'$-contractible edge of $H'$ and recall that $\Omega'$ can be assumed to be a witness for $e'$.
By the definition of $\rho'$-contractible edge, there exists $\omega^{(e')}\in\Omega'$ and an isomorphism $\phi^{(e')}$ from $H'[\omega^{(e')}_{[t'-1]}]$ to $H'/e'$ such that $\phi^{(e')}(\rho'_i)=\rho'_i$ for every $i \in [\min\{m',t'-1\}]$.

Define $\phi^*(v)=v$ for $v \in V(H)$ and define $\phi^*(v)=\phi^{(e')}(v)$ for $v \in \omega^{(e')}_{[t'-1]}$.
Note that $\phi^*$ is well-defined since $V(H) \cap \omega^{(e')}_{[t'-1]} = \{\rho'_i: i \in [\min\{m',t'-1\}]\}$ and $\phi^{(e')}(\rho'_i)=\rho'_i$ for $i\in[\min\{m',t'-1\}]$.
Let $\omega^* \in \Omega$ be arbitrary and define $\psi^*=\psi_{\omega^*,\omega^{(e')}}$.
Then $\phi^*$ is an isomorphism from $H \cup H'[\omega'_{[t'-1]}] = G[\psi^*_{[\tau-1]}]$ to $H \cup (H'/e') = G/e'$ such that $\phi^*(\rho_i)=\rho_i$ for every $i \in [\min\{m,\tau-1\}]$.
This shows that $e'$ is a $\rho$-contractible edge in $G$, and completes the proof of the lemma.
\end{proof}

Note that Proposition \ref{prop:tree} also follows from Lemma \ref{lem:attachingroots}.
Now we show another example that can be deduced from Lemma \ref{lem:attachingroots}.
Note that if $H$ is a maximal outerplanar graph with $|V(H)|\geq 4$ whose weak dual tree is a path, then $H$ has exactly two vertices of degree 2 and they are non-adjacent.

\begin{proposition} \label{prop:weakdualpath}
Let $H$ be a maximal outerplanar graph with at least three vertices whose weak dual tree is a path.
Let $uv$ be an edge such that $u$ or $v$ has degree 2 in $H$. 
Let $e$ be an edge of $H$ incident to a vertex $z\not\in\{u,v\}$ of degree two in $H$. 
Then $e$ is a $(u,v)$-contractible edge of $H$ with witness $\Omega$ for some $(u,v)$-rooted contractible ordering $\Omega$. 
\end{proposition}

\begin{proof}
We proceed by induction on $|V(H)|$.
Note that for every graph isomorphic to $K_3$ with vertices $a,b,c$, every edge in $\{bc,ac\}$ is an $(a,b)$-contractible edge.
This implies that the proposition holds when $|V(H)|=3$.
So we may assume that $|V(H)| \geq 4$ and that the proposition holds for graphs with fewer than $|V(H)|$ vertices.

Since the weak dual tree of $H$ is a path and $z$ has degree two in $H$, $H-z$ is a maximal outerplanar graph whose weak dual is a path.
Let $a,b$ be the neighbors of $z$ in $H$.
Since $|V(H)| \geq 4$ and $H$ is an outerplanar graph whose weak dual tree is a path, we have that $ab$ is an edge of $H-z$ different from $e$ incident with a vertex of degree two in $H-z$.
Since $|V(H-z)| \geq 3$, by the induction hypothesis, $ab$ is a $(u,v)$-contractible edge of $H-z$.
Since $H[\{a,b,z\}]$ is isomorphic to $K_3$, $e$ is an $(a,b)$-contractible edge of $H[\{a,b,z\}]$.
Hence $e$ is a $(u,v)$-contractible edge of $H$ by Statement 3 of Lemma \ref{lem:attachingroots}.
\end{proof}

A \emph{cactus} is a connected graph where every block is either an edge or a cycle.

\begin{proposition} \label{prop:cactus}
Let $G$ be a graph and $r \in V(G)$.
If $G$ is a cactus, then $G$ is a spanning subgraph of a graph that admits a $(r)$-rooted contractible ordering.
\end{proposition}

\begin{proof}
Let $G$ be a cactus.
We proceed by induction on the number of blocks of $G$.
We first assume that $G$ has exactly one block.
The proposition obviously holds when $|V(G)|=1$.
So $G$ is either an edge or a cycle. 
Let $v\in V(G)$ be a neighbor of $r$, and let $H$ be the graph obtained from $G$ by adding an edge joining $v$ to every vertex in $V(G)\setminus\{v\}$ that is not already adjacent to $v$.
By Propositions \ref{prop:tree} or \ref{prop:weakdualpath}, $H$ admits a $(r)$-rooted contractible ordering.

We may thus assume that $G$ has at least two blocks.
Let $B$ be a leaf block of $G$. 
Since $G$ is connected, there exists the unique vertex $b$ in $V(B)$ adjacent in $G$ to $V(G)\setminus V(B)$.
Let $G'=G\setminus (V(B)\setminus \{b\})$.
Since $G$ has at least two blocks, we may choose $B$ so that $r \in V(G')$.
Since $G'$ is a connected cactus with fewer blocks than $G$, by the inductive hypothesis, $G'$ is a spanning subgraph of a graph $H'$ that admits a $(r)$-rooted contractible ordering.
Since $B$ is a connected cactus with fewer blocks than $G$, by the inductive hypothesis, $B$ is a spanning subgraph of a graph $B'$ that admits a $(b)$-rooted contractible ordering.
In particular, $B'$ has a $(b)$-rooted contractible edge and $V(H') \cap V(B')=\{b\}$.
So by Statement 2 of Lemma \ref{lem:attachingroots}, $H' \cup B'$ admits a $(r)$-rooted contractible ordering.
This proves the proposition since $G$ is a spanning subgraph of $H' \cup B'$.
\end{proof}

\section{Digraphs}\label{sec:digraphs}

Let $\vec G$ be a digraph.
A {\it directed separation} of $\vec G$ is an ordered pair $(A,B)$ with $A \cup B=V(\vec G)$ such that there does not exist an arc in $\vec G$ from $B-A$ to $A-B$; 
$(A,B)$ is \emph{nontrivial} if $A\setminus B$ and $B\setminus A$ are both nonempty, and its \emph{order} is defined to be $|A\cap B|$.
We say that a subset $S$ of $V(\vec G)$ is \emph{well-in-connected in $\vec G$} if $S\neq V(\vec G)$ and there does not exist a nontrivial directed separation $(A,B)$ in $\vec G$ with $S\subseteq A$ and $|A\cap B|<|S|$.

Recall that, for a digraph $\vec H$, $\vec H^+$ denotes the digraph obtained from $\vec H$ by adding an apex source (a new vertex and new arcs from the new vertex to every other vertex).

\subsection{In-arborescences with an apex source} \label{sec:apexarborescence}
Recall that an in-arborescence is a directed tree with a vertex $r$, called its root, such that every arc is oriented towards the root. 
A \emph{leaf} of an in-arborescence is a vertex with in-degree 0.

We prove Theorem \ref{thm:butterfly_deg} by proving the following more technical lemma, which is analogous to Lemma \ref{lem:contractableorder} except the proof here is simpler because, when adding a new vertex to $S$, the new vertex has only one out-neighbor in $S$ (whereas in Lemma \ref{lem:contractableorder} the new vertex could have up to two neighbors in $S$).
\begin{lemma}\label{lem:apexarbor}
Let $t$ be a positive integer and let $\vec H$ be an in-arborescence on $t$ vertices with root $r$.  
Let $\vec G$ be a digraph.
Let $S \subseteq V(\vec G)$ be well-in-connected in $\vec G$ such that $\vec G[S]$ is isomorphic to a weakly connected subdigraph of $\vec H$ containing $r$.
If every vertex in $\vec G\setminus S$ has out-degree at least $t$, then $\vec G$ contains $\vec H^+$ as a butterfly minor such that the branch set of the apex source is a singleton.
\end{lemma}

\begin{proof}
Suppose to the contrary and let $(\vec G,S)$ be a counterexample minimizing $|V(\vec G)\setminus S|$.
Let $\vec H'$ be a weakly connected subdigraph of $\vec H$ containing $r$ isomorphic to $\vec G[S]$.

Since $\vec H'$ contains $r$, we have $S \neq \emptyset$.
If $|S|=t$, then $\vec H'=\vec H$ and, since $S$ is well-in-connected in $\vec G$, there exist a vertex $v\in V(\vec G)\setminus S$ and $t$ directed paths in $\vec G$ from $v$ to $S$ only sharing $v$ by Menger's theorem. 
By contracting these paths (towards $S$) into single arcs, we obtain $\vec H^+$ as a butterfly minor of $\vec G$ such that the branch set of the apex source is the singleton $\{v\}$.

So we may assume that $0<|S|<t$. 

    \begin{claim} \label{lem:apexarbor:claim1}
        There does not exist a nontrivial directed separation $(A,B)$ in $\vec G$ such that $S\subseteq A$ and $|A\cap B|\leq |S|$.
    \end{claim}

    \begin{subproof}[Proof of Claim \ref{lem:apexarbor:claim1}]
        Suppose to the contrary that such a directed separation $(A,B)$ exists.
        Since $S$ is well-in-connected in $\vec G$, we have $|A \cap B|=|S|$ and, by Menger's theorem, there exist $|S|$ disjoint directed paths $P_1,\dots,P_{|S|}$ from $A\cap B$ to $S$. 
        Let $\vec G'$ be the digraph obtained from $\vec G$ by  
        \begin{itemize}
            \item deleting every vertex in $A\setminus (S \cup \bigcup_{i=1}^{|S|}V(P_i)) = A \setminus \bigcup_{i=1}^{|S|}V(P_i)$ and every arc $(u,v) \in E(\vec G) \setminus E(\vec G[S] \cup \bigcup_{i=1}^{|S|}P_i)$ with $u \in \bigcup_{i=1}^{|S|}V(P_i)$, and
            \item contracting each directed path $P_i$ into a single vertex $s_i$.
        \end{itemize}
        Note that $\vec G'$ is a butterfly minor of $\vec G$ since, after the first step,  every vertex in $A\setminus S$ has out-degree 1 in $\vec G'$. 
        Now let $S' = \{s_1,\dots,s_{|S|}\}$.
        Then $S'$ is well-in-connected in $\vec G'$ since, if $(A',B')$ is a nontrivial directed separation of $\vec G'$ with $S'\subseteq A'$ and $|A' \cap B'|<|S'|$, then $((A'\setminus S')\cup A,B')$ is a nontrivial directed separation in $\vec G$ with $S\subseteq A \subseteq (A'\setminus S')\cup A$ and with order $|((A'\setminus S')\cup A)\cap B'| = |A' \cap B'|<|S'| = |S|$, contradicting the well-in-connectedness of $S$ in $\vec G$.
        Also note that $\vec G'[S']$ is isomorphic to $\vec G[S]$, which is isomorphic to $\vec H'$,  a weakly connected subdigraph of $\vec H$ containing $r$.
        Moreover, every vertex in $\vec G'\setminus S'$ is a vertex in $\vec G\setminus S$ and has out-degree in $\vec G'$ equal to that in $\vec G$, which is at least $t$. 
        Since $A \setminus B \neq \emptyset$, we have $|V(\vec G') \setminus S'|<|V(\vec G) \setminus S|$.
        Thus, $\vec H^+$ is a butterfly minor of $\vec G'$ by the minimality of $|V(\vec G)\setminus S|$.
        Since $\vec G'$ is a butterfly minor of $\vec G$, it follows that $\vec H^+$ is a butterfly minor of $\vec G$, a contradiction.
    \end{subproof}
    
    Since $0<|S|<t$, there exists a vertex $u'\in V(\vec H)\setminus V(\vec H')$ and a vertex $v'\in V(\vec H')$ such that $(u',v')\in E(\vec H)$. 
    Let $v$ denote the vertex in $\vec G[S]$ corresponding to $v'$. 
    Then there is a vertex $u\in V(\vec G)\setminus S$ such that $(u,v)\in E(\vec G)$; otherwise, $(S,V(\vec G)\setminus \{u\})$ is a nontrivial directed separation of $\vec G$ with order $|S \cap (V(\vec G) \setminus \{u\})| = |S|-1$, contradicting the well-in-connectedness of $S$ in $\vec G$.
    
    Let $S' = S\cup\{u\}$ and let $\vec G'$ denote the digraph obtained from $\vec G$ by deleting all outgoing arcs of $u$ except $(u,v)$.
    Then $\vec G'[S']$ is isomorphic to the digraph $(V(\vec H')\cup\{u'\},E(\vec H')\cup\{(u',v')\})$, which is a weakly connected subdigraph of $\vec H$ containing $r$.
    Moreover, $S'$ is well-in-connected in $\vec G'$: if $(A,B)$ is a nontrivial directed separation of $\vec G'$ with $S'\subseteq A$ and with $|A \cap B|<|S'|$, then $(A,B)$ is also a nontrivial directed separation in $\vec G$ with $S\subseteq A$ and $|A \cap B| \leq |S|$, contradicting Claim \ref{lem:apexarbor:claim1}. 
    Finally, every vertex in $\vec G'\setminus S'$ has out-degree in $\vec G'$ equal to that in $\vec G$, which is at least $t$. 
    Since $|V(\vec G')\setminus S'| = |V(\vec G)\setminus S|-1$, we have that $\vec H^+$ is a butterfly minor of $\vec G'$ by the minimality of $|V(\vec G)\setminus S|$. Hence $\vec H^+$ is also a butterfly minor of $\vec G$, a contradiction.
\end{proof}

Theorem \ref{thm:butterfly_deg}, restated below, now follows from Lemma \ref{lem:apexarbor}.

\butterflydeg*

\begin{proof}
Let $t=|V(\vec H)|$ and let $\vec G$ be a digraph with minimum out-degree at least $t$. 
Note that $\vec G$ has a strongly connected component $Q$ such that for every $v\in V(Q)$, the out-degree of $v$ in $Q$ is equal to the out-degree of $v$ in $\vec G$, which is at least $t$. 
Let $v$ be an arbitrary vertex in $Q$, and let $S = \{v\}$. 
Since $Q$ is strongly connected, we have that $S$ is well-in-connected in $Q$. 
Since $Q[S]$ consists of a single vertex, it is isomorphic to the weakly connected subdigraph of $\vec H$ consisting of the root. 
Moreover, every vertex in $Q \setminus S$ has out-degree at least $t$. 
It follows from Lemma \ref{lem:apexarbor} that $Q$ (and hence $\vec G$) contains $\vec H^+$ as a butterfly minor such that the branch set of the apex source is a singleton.
\end{proof}

\subsection{Directed wheels}\label{sec:directedwheel}

Let $S$ be a subset of $V(\vec G)$.
We define $\partial^-_{\vec G}(S)$ to be the set $\{v\in S: (u,v)\in E(\vec G)$ for some $u\in V(\vec G)\setminus S\}$. 
If the digraph $\vec G$ is clear from context, we simply write $\partial^-(S)$.

We prove Theorem \ref{thm:subdiv_wheel_mid_deg_intro} by proving the following technical lemma.

\begin{lemma} \label{wheel_directed_lemma}
Let $t \geq 2$ be an integer.
Let $\vec G$ be a strongly connected digraph and let $S$ be a subset of $V(\vec G)$ such that 
\begin{itemize}
    \item every vertex in $V(\vec G) \setminus S$ has out-degree at least $t$ in $\vec G$,
    \item there exist a directed cycle $C$ in $\vec G[S]$ and $|\partial^-(S)|$ disjoint directed paths in $\vec G$ from $\partial^-(S)$ to $V(C)$ internally disjoint from $V(C)$, and
    \item there exists an arc $(x,y)\in E(\vec G)$ such that $x\in V(C)$ and $y \in V(\vec G)\setminus S$.
\end{itemize}
   Then $\vec G$ contains a subdivision of $\vec W_t^1$ or a subdivision of $\vec W_{t+1}^2$.
\end{lemma}

\begin{proof}
Suppose to the contrary and choose a counterexample $(\vec G,S)$ so  that $|V(\vec G) \setminus S|$ is minimum. 
Let $P_1,P_2,...,P_{|\partial^-(S)|}$ denote the $|\partial^-(S)|$ disjoint directed paths from $\partial^-(S)$ to $V(C)$ internally disjoint from $V(C)$.

\begin{claim} \label{lem:outwheel:claim1}
    There does not exist a directed separation $(A,B)$ of $\vec G$ with $|A \cap B| < |\partial^-(S)|$ such that $S\subseteq A$ and $y \in B \setminus A$.
\end{claim}

\begin{subproof}[Proof of Claim \ref{lem:outwheel:claim1}]
Suppose to the contrary that such a directed separation $(A,B)$ exists.
We choose $(A,B)$ so that $|A \cap B|$ is minimum.

If there exists $v \in A \cap B$ such that there is no arc of $\vec G$ from $B \setminus A$ to $v$, then $(A,B \setminus \{v\})$ is a directed separation of $\vec G$ with $|A \cap (B \setminus \{v\})|<|A \cap B| < |\partial^-(S)|$ such that $S\subseteq A$ and $y \in (B\setminus \{v\})\setminus A \subseteq B\setminus A$, 
contradicting the minimality of $|A \cap B|$.

So for every $v \in A \cap B$, there exists an arc of $\vec G$ from $B \setminus A$ to $v$.
That is, $A \cap B = \partial^-(A)$.
If there does not exist $|A \cap B|$ disjoint directed paths in $\vec G[A]$ from $A \cap B$ to $S$ internally disjoint from $S$, then by Menger's theorem, there exists a directed separation $(A',B')$ of $\vec G[A]$ with $|A' \cap B'|<|A \cap B|$ such that $S \subseteq A'$ and $A \cap B \subseteq B'$, so $(A',B' \cup B)$ is a directed separation of $\vec G$ with order $|A' \cap (B' \cup B)|=|A' \cap B'|<|A \cap B| < |\partial^-(S)|$ such that $S\subseteq A'$ and $y\in B\setminus A\subseteq (B'\cup B)\setminus A'$, 
contradicting the minimality of $|A \cap B|$.

So there exist $|A \cap B|$ disjoint directed paths $Q_1,Q_2,...,Q_{|A \cap B|}$ in $\vec G[A]$, internally disjoint from $S$, from $A \cap B$ to $S$ (hence to $\partial^-(S)$).
By possibly reindexing, we may assume that for each $i \in [|A \cap B|]$, the vertex in $V(Q_i) \cap \partial^-(S)$ is the vertex in $V(P_i) \cap \partial^-(S)$.
Hence $P_1 \cup Q_1, P_2 \cup Q_2,..., P_{|A \cap B|} \cup Q_{|A \cap B|}$ are disjoint directed paths in $\vec G[A]$ from $A \cap B$ to $V(C)$ internally disjoint from $V(C)$.

Let $S'=A$.
Then $S'\subseteq V(\vec G)$ and $\partial^-(S')=\partial^-(A)=A \cap B$.
Note that 
every vertex in $V(\vec G) \setminus S' \subseteq V(\vec G) \setminus S$ has out-degree in $\vec G$ at least $t$.
Moreover, since $V(C) \subseteq S \subseteq A=S'$, we have that $C$ is a directed cycle in $\vec G[S']$ 
and that there are $|A \cap B|=|\partial^-(S')|$ disjoint directed paths in $\vec G$ from $A \cap B=\partial^-(S')$ to $V(C)$ internally disjoint from $V(C)$.
Furthermore, $(x,y)$ is an arc with $x\in V(C)$ and $y\in B\setminus A \subseteq V(\vec G)\setminus S'$.

On the other hand, note that $S\subsetneq S'$: since $S\subseteq A=S'$, if $S\supseteq S'$, then $S=A=S'$, so by the definition of $(A,B)$ we have $|A \cap B|<|\partial^-(S)|$, which gives $|\partial^-(A)|=|A \cap B|<|\partial^-(S)|=|\partial^-(A)|$, a contradiction.
Therefore, $|V(\vec G)\setminus S'|<|V(\vec G)\setminus S|$, and it follows by the minimality of $|V(\vec G)\setminus S|$ that $\vec G$ contains a subdivision of $\vec W_t^1$ or $\vec W_{t+1}^2$, a contradiction.
\end{subproof}

\begin{claim}\label{lem:outwheel:claim2}
    There exist $|\partial^-(S)|$ directed paths in $\vec G[(V(\vec G) \setminus S) \cup \partial^-(S)]$ from $y$ to $\partial^-(S)$ internally disjoint from $S$ only sharing $y$.
\end{claim}

\begin{subproof}[Proof of Claim \ref{lem:outwheel:claim2}]
If there do not exist $|\partial^-(S)|$ directed paths in $\vec G$ from $y$ to $S$ only sharing $y$, then by Menger's theorem, there exists a directed separation $(A,B)$ of $\vec G$ with $|A \cap B|<|\partial^-(S)|$ such that $S \subseteq A$ and $y \in B \setminus A$, contradicting Claim \ref{lem:outwheel:claim1}.
So there exist $|\partial^-(S)|$ paths in $\vec G$ from $y$ to $S$ only sharing $y$.
By taking subpaths, we may assume that these paths are internally disjoint from $S$, hence they end at $\partial^-(S)$ and are contained in $\vec G[(V(\vec G) \setminus S) \cup \partial^-(S)]$.
\end{subproof}

Let $Q_1,Q_2,...,Q_{|\partial^-(S)|}$ be $|\partial^-(S)|$ directed paths in $\vec G[(V(\vec G)-S) \cup \partial^-(S)]$ as in Claim \ref{lem:outwheel:claim2}, chosen so that $\sum_{i=1}^{|\partial^-(S)|}|V(Q_i)|$ is minimum.
This implies that for every $i \in [|\partial^-(S)|]$, exactly one out-neighbor of $y$ in $\vec G$ is contained in $V(Q_i)$.

By possibly reindexing, we may assume that for each $i \in [|\partial^-(S)|]$, we have $V(P_i) \cap \partial^-(S) = V(Q_i) \cap \partial^-(S)$. 

\begin{claim}\label{lem:outwheel:claim3}
The vertex $y$ has an out-neighbor $y'\in V(\vec G) \setminus (S \cup \bigcup_{i=1}^{|\partial^-(S)|}V(Q_i))$. 
\end{claim}

\begin{subproof}[Proof of Claim \ref{lem:outwheel:claim3}]
Note that $\{Q_i \cup P_i: i \in [|\partial^-(S)|]\}$ is a set of $|\partial^-(S)|$ directed paths in $\vec G$ from $y$ to $V(C)$ only sharing $y$.
If $|\partial^-(S)| \geq t$, then these directed paths together with $C$ and the arc $(x,y)$ form a subdivision of $\vec W_t^1$ or $\vec W_{t+1}^2$, a contradiction.

So we may assume $|\partial^-(S)| \leq t-1$.
Since $y \in V(\vec G) \setminus S$, the out-degree of $y$ in $\vec G$ is at least $t$.
Since for every $i \in [|\partial^-(S)|]$, exactly one out-neighbor of $y$ in $\vec G$ is contained in $V(Q_i)$, there exists an out-neighbor $y'$ of $y$ in $\vec G$ not contained in $\bigcup_{i=1}^{|\partial^-(S)|}V(Q_i)$.
Note that $y'$ is also not contained in $S$ since  $y\in V(\vec G)\setminus S$ and $\bigcup_{i=1}^{|\partial^-(S)|}V(Q_i) \supseteq \partial^-(S)$.  
Hence $y' \in V(\vec G) \setminus (S \cup \bigcup_{i=1}^{|\partial^-(S)|}V(Q_i))$.
\end{subproof}

Let $z$ be the vertex in $V(C) \cap \bigcup_{i=1}^{|\partial^-(S)|}V(P_i)$ such that $x \neq z$ and no internal vertex of the directed path in $C$ from $x$ to $z$ is in $\bigcup_{i=1}^{|\partial^-(S)|}V(P_i)$.
Let $P'$ denote the directed path in $C$ from $z$ to $x$. 
Let $j\in [|\partial^-(S)|]$ be the index such that $P_{j}$ contains $z$. Then $C':=P'\cup \{(x,y)\} \cup Q_{j}\cup P_{j}$ is a directed cycle.

Let $S' = S\cup V(Q_j)$.
Then $S'\subseteq V(\vec G)$, and every vertex in $V(\vec G)\setminus S'$ is also in $V(G)\setminus S$ and hence has out-degree at least $t$ in $\vec G$.
We also have that $C'$ is a directed cycle in $\vec G[S']$.

Moreover, we have $\partial^-(S') \subseteq \partial^-(S) \cup (S' \setminus S) \subseteq \partial^-(S) \cup V(Q_{j})$.
Since $V(Q_{j}) \subseteq V(C')$, each vertex $v$ in $V(Q_{j})$ forms a trivial directed path $P_v$ from $V(Q_{j})$ to $V(C')$.
The paths $P_i$ for $i \in [|\partial^-(S)|] \setminus \{j\}$ are disjoint paths from $\partial^-(S)\setminus V(Q_j)$ to $V(C')$.
Hence $\mathcal{P}:=\{P_v: v \in V(Q_{j})\} \cup \{P_i: i \in [|\partial^-(S)|] \setminus \{j\}\}$ is a set of $|V(Q_j)\cup \partial^-(S)|$ disjoint directed paths in $\vec G$ from $\partial^-(S) \cup V(Q_{j})$ to $V(C')$, and since $\partial^-(S') \subseteq \partial^-(S) \cup V(Q_{j})$, it follows that $\mathcal{P}$ contains $|\partial^-(S')|$ disjoint directed paths in $\vec G$ from $\partial^-(S')$ to $V(C')$ internally disjoint from $V(C')$.

Furthermore, by Claim \ref{lem:outwheel:claim3}, there exists an arc $(y,y')\in E(\vec G)$ such that $y\in V(C')$ and $y'\in V(\vec G)\setminus S'$.

Since $y \not\in S$, we have $S\subsetneq S'$ and hence $|V(\vec G) \setminus S'|<|V(\vec G) \setminus S|$.
By the minimality of $|V(\vec G) \setminus S|$ among counterexamples, $\vec G$ contains a subdivision of $\vec W_t^1$ or $\vec W_{t+1}^2$, a contradiction.
This proves the lemma.
\end{proof}

Theorem \ref{thm:subdiv_wheel_mid_deg_intro}, restated below, now follows from Lemma \ref{wheel_directed_lemma}.
\subdivwheel*

\begin{proof}
Note that $\vec G$ has a strongly connected component $Q$ such that for every $v \in V(Q)$, the out-degree of $v$ in $Q$ is equal to the out-degree of $v$ in $\vec G$, which is at least $t$.
Since $Q$ is strongly connected and $t \geq 2$, there exists a directed cycle $C$ in $Q$.
We choose $C$ such that $|E(C)|$ is as small as possible.

Let $x \in V(C)$.
Then there exists $(x,y) \in E(Q)$ with $y \in V(Q) \setminus V(C)$, for otherwise, since $t \geq 2$, there exists an arc of $Q$ from $x$ to another vertex $v$ in $C$ such that $(x,v) \not \in E(C)$, so the arc $(x,v)$ together with the subpath of $C$ from $v$ to $x$ gives a directed cycle with fewer arcs than $C$, a contradiction.

Let $S=V(C)$.
Then $C$ is a directed cycle in $Q[S]$, and there exist $|\partial_Q^-(S)|$ disjoint directed paths from $\partial_Q^-(S)$ to $V(C)$, each having exactly one vertex.
Moreover, every vertex of $Q$ has out-degree at least $t$ in $Q$, and there exists $(x,y) \in E(Q)$ with $y \in V(Q) \setminus S$.
By Lemma \ref{wheel_directed_lemma}, the digraph $Q$ (and hence $\vec G$) contains a subdivision of $\vec W_t^1$ or a subdivision of $\vec W_{t+1}^2$.
Since $\vec C_t^+$ and $\vec W_t^2$ are subgraphs of both $\vec W_t^1$ and $\vec W_{t+1}^2$, we know that $\vec G$ contains a subdivision of $\vec C_t^+$ and a subdivision of $\vec W_t^2$.
\end{proof}

\section{Concluding remarks}
\label{sec:conclusion}
In this paper we studied undirected and directed graphs $H$ for which every undirected and directed graph with minimum degree and out-degree at least $|V(H^+)|-1=|V(H)|$ contains $H^+$ as a minor and a butterfly minor respectively, and as a subdivision.
All of our results are tight since $K_{|V(H^+)|-1}$ and its biorientation have minimum degree and out-degree $|V(H^+)|-2$ and yet do not contain any graph on $|V(H^+)|$ vertices as a minor, a butterfly minor, or a subdivision.

For undirected graphs, since $K_4$ and $K_{2,3}$ are negative answers to Question \ref{q:apexmindegree}, we propose the following conjecture.

\begin{conjecture}\label{conj:apexouterplanar}
If $H$ is an outerplanar graph, then every graph with minimum degree at least $|V(H)| = |V(H^+)|-1$ contains $H^+$ as a minor. 
\end{conjecture}

Theorem \ref{thm:mainmindeg} shows that a large class of outerplanar graphs satisfy Conjecture \ref{conj:apexouterplanar}. 
Since every outerplanar graph is 2-degenerate, every apex-outerplanar graph is 3-degenerate. 
Hence Conjecture \ref{conj:apexouterplanar} is implied by the following conjecture of Diwan \cite{diwan1971} on subdivisions.

\begin{conjecture}[{\cite[Conjecture 1]{diwan1971}}]
If $H$ is a planar maximal 3-degenerate graph, then every graph with minimum degree at least $|V(H)|-1$ contains a subdivision of $H$.
\end{conjecture}

In fact, we are not aware any example of a planar graph $H$ which does not satisfy Question \ref{q:mindegree}, or Question \ref{q:subdivmindeg}.

For digraphs, Mader's conjecture (Conjecture \ref{conj:mader}) remains open, and several related problems are given in \cite{aboulker2019subdivisions, gishboliner2022dichromatic, gishboliner2022oriented}. 
We conclude with another conjecture on orientations of cycles.
Aboulker et al.\ conjectured \cite[Conjecture 27]{aboulker2019subdivisions} that every digraph with minimum out-degree at least $2t-1$ contains every orientation of the cycle on $t$ vertices as a subdivision.
This conjecture remains open.
Gishboliner et al.\ \cite{gishboliner2022oriented} proved that there is a polynomial function $f$ such that every digraph with minimum out-degree at least $f(t)$ satisfies this conclusion.
The same set of authors \cite{gishboliner2022dichromatic} also proved that every digraph with dichromatic number at least $t$ contains every orientation of the cycle on $t$ vertices as a subdivision.
Since every digraph with dichromatic number at least $t$ contains a subdigraph with minimum out-degree and minimum in-degree at least $t-1$, we propose the following common strengthening of those results.

\begin{conjecture}
    Every digraph with minimum out-degree and minimum in-degree at least $t-1$ contains every orientation of the cycle on $t$ vertices as a subdivision.
\end{conjecture}

\bigskip

\bigskip

\noindent{\bf Acknowledgement:}
This work was initiated when the first author participated Second 2022 Barbados Graph Theory Workshop.
He thanks the participants of the workshop, especially Sepehr Hajebi, Paul Seymour, Zixia Song and Raphael Steiner, for discussions.
This paper was partially written when the first author visited the National Center for Theoretical Sciences Mathematics Division in Taiwan and the Institute of Mathematics at Academia Sinica in Taiwan.

\bibliography{ref} 
\bibliographystyle{habbrv}

\end{document}